\documentclass[preprint,12pt]{imsart}
\usepackage[OT1]{fontenc}
\usepackage[utf8]{inputenc}
\usepackage[colorlinks]{hyperref}
\usepackage{amsmath,amssymb,amsthm}
\usepackage{mathrsfs}
\usepackage[english]{babel}

\usepackage{graphicx}
\usepackage{color}
\usepackage{indentfirst}
\usepackage{multirow}
\usepackage{float}
\usepackage{subfig}
\usepackage{a4wide}
\usepackage[authoryear]{natbib}

\startlocaldefs
\numberwithin{equation}{section}
\theoremstyle{plain}
\newtheorem{theorem}{Theorem}[section]
\newtheorem{lemma}{Lemma}[section]

\newtheorem{proposition}{Proposition}[section]
\newtheorem{remark}{Remark}[section]
\endlocaldefs

\begin{document}


\newcommand\tr{\mathop{\text{tr}}}
\newcommand\e{\mathbb{E}}
\newcommand\var{\mathbb{V\text{ar}}}
\newcommand\diag{\mathop{\text{diag}}}
\newcommand\bl[1]{{\color{blue}#1}}
\newcommand{\ga}{\gamma}
\newcommand{\la}{\lambda}
\newcommand{\eps}{\varepsilon}
\newcommand\cov{\text{Cov}}
\newcommand\E{\mathbb{E}}
\newcommand\R{\mathbb{R}}

 \begin{frontmatter}
   \title{Extreme eigenvalues of large-dimensional spiked Fisher
     matrices with application}
   \runtitle{Large and spiked Fisher matrices with application}

   \begin{aug}
     \author{\fnms{Qinwen} \snm{Wang}\ead[label=e1]{wqw8813@gmail.com}}
     \and
     \author{\fnms{Jianfeng} \snm{Yao}\ead[label=e3]{jeffyao@hku.hk}}

     \runauthor{Q. Wang and J. Yao}

     \affiliation{Zhejiang  University and The University of Hong Kong}

     \address{Qinwen Wang  \\
     Department of Mathematics\\
     Zhejiang University \\
     \printead{e1}
   }

     \address{Jianfeng Yao \\
     Department of Statistics and Actuarial Science\\
     The University of Hong Kong\\
     Pokfulam, \quad
     Hong Kong \\
     \printead{e3}
   }
   \end{aug}

   \begin{abstract}
     Consider two $p$-variate  populations, not necessarily Gaussian,   with
covariance matrices $\Sigma_1$ and $\Sigma_2$,  respectively,
and let $S_1$ and $S_2$ be the sample covariances matrices from
samples of the populations with degrees of freedom  $T$ and $n$,
respectively.
When the difference $\Delta$  between $\Sigma_1$ and $\Sigma_2$ is of
small rank compared to $p,T$ and $n$,
the Fisher matrix $F=S_2^{-1}S_1$ is called a {\em spiked Fisher matrix}.
When  $p,T$ and $n$  grow to
infinity proportionally,
we establish a phase transition for the extreme eigenvalues of
$F$:
when the eigenvalues of $\Delta$ ({\em spikes}) are  above (or under)
a critical value,  the associated  extreme eigenvalues of the Fisher
matrix will converge to some point outside the support of the global
limit (LSD) of other eigenvalues;
otherwise, they  will converge to the edge points of the LSD.
Furthermore, we derive
central limit theorems  for these
extreme eigenvalues of the spiked Fisher matrix. The limiting
distributions are found to be
Gaussian if and only if the corresponding population
spike eigenvalues in $\Delta$ are  {\em simple}.
Numerical examples 
are provided to demonstrate   the finite sample performance of the
results. In addition to classical applications of a Fisher matrix in
high-dimensional data analysis, we propose a new method for the
detection of signals  allowing an arbitrary covariance structure of the
noise.  Simulation experiments are conducted to illustrate the
performance of  this detector.
\end{abstract}


 \end{frontmatter}

\section{Introduction}
\label{intro}

Consider two $p$-variate populations with
covariance matrices $\Sigma_1$ and $\Sigma_2$,  respectively,
and let $S_1$ and $S_2$ be the sample covariances matrices from
samples of the populations with degrees of freedom  $T$ and $n$,
respectively.
Specifically,  if both populations are Gaussian, $TS_1$ and $nS_2$
are distributed  as  Wishart  $W_P(T,\Sigma_1)$ and $W_P(n,\Sigma_2)$,
respectively. For testing the equality  hypothesis
$H_0: ~\Sigma_1=\Sigma_2$,  the likelihood ratio  statistic relies on
the $p$ characteristic roots of the determinental equation
\begin{align}
  \label{eq:det}
  |S_1 - l S_2|=0~, \quad l\in\R.
\end{align}
Here and throughout the paper, the determinant of a matrix $A$ is
denoted by either $|A|$ or $\det(A)$. 
As  a  famous story in
multivariate analysis of  last century,  the joint distribution of these
characteristic roots for Gaussian populations
was  simultaneously and independently
published  in 1939  by R. A. Fisher, S. N. Roy, P. L. Hsu and
M. A. Girshick.
When $S_2$ is invertible, these roots are simply the eigenvalues of
the matrix $F=S_2^{-1}S_1$, widely known  as a {\em Fisher matrix}
in the literature, which generalises the one-dimensional Fisher
ratio.

Another breakthrough is the work of \citet{Wachter80}
where he finds a deterministic limit, the celebrated Wacheter distribution,
for the empirical measure  of these roots when the
dimension $p$ grows to infinity
proportionally to the degrees of freedom  $T$ and  $n$ (under the
Gaussian assumption).
Wachter's result has been later
extended to non-Gaussian populations in what is now  called  the random
matrix theory and  two early examples of such extensions  are
\cite{Silverstein85} and \cite{BaiYinKrish87} .
It is also important to notice that the determinental equation
\eqref{eq:det}
arises not only in the classical hypothesis testing problem mentioned
above, it indeed covers also similar equations arising in important
fields of multivariate analysis such as discriminant
analysis, canonical correlation analysis and MANOVA, see
\citet{Wachter80}.

Needless to say that such limiting  results allowing large values of
dimension  $p$ comparable to the degrees of freedom (i.e. sample sizes)
are going to  have much impact on today's high-dimensional data analysis.
A particularly important question is to
investigate the properties of the characteristic roots
under an alternative of form
\begin{equation}
  \label{eq:H1}
  H_1: ~~\Sigma_1 = \Sigma_2 + \Delta~,
\end{equation}
where $\Delta$ is a nonnegative definite matrix of rank $M$.
When $p$, $T$ and $n$ are all large, the discrimination between the
null hypothesis and the alternative is not difficult if the rank
difference $M$ is all large.
The real challenge here lies in detecting a  small rank-$M$  alternative.
In this perspective and assuming $M$ is a fixed integer while
$p$, $T$ and $n$ grow to infinity proportionally,
the empirical measure of the $p$ characteristic roots of
\eqref{eq:det} will  be affected by a difference of order
$M/p$ which vanishes, so that its limit remains the same as in the
null hypothesis, i.e. the Wachter distribution.
In other words,  such global limit from all the
characteristic roots will be of little help  for distinguishing the
two hypotheses.

It happens  that the useful information to detect a small rank
alternative is encoded in a few largest characteristic roots of
\eqref{eq:det}.
In a recent preprint  \cite{DhaJoh14}, by  assuming
both population are Gaussian  and $M=1$,
these authors
show that, when the norm of the  rank-1 difference $\Delta$
({\em spike})
exceeds
a phase transition threshold,
the
asymptotic behaviour of the log-ratio of the joint density of these
characteristic
roots  under a local deviation from the spike  depends only on
the largest characteristic root $l_{p,1}$ and
the statistical experiment of observing all the characteristic
roots is locally asymptotically normal (LAN).
As a by-product of their analysis, the authors
also  establish joint asymptotic normality of a few of
the largest roots
when the corresponding spikes in $\Delta$ (with $M>1$)
exceed  the phase transition threshold.
As it can be guessed, the analysis given in this reference  highly rely on
the Gaussian assumption so that  the joint density function of the
characteristic roots has indeed an explicit form under both the null
and  the
alternative,  and the main results are obtained via an accurate
analytic approximation of the log-ratio of these density functions
when the dimension $p$, $T$ and $n$ grow to infinity proportionally.

Intrigued by these findings,  in this paper, we  explore the same
questions for  {\em general populations without Gaussian assumption}.
It is thus  apparent that the joint density of the characteristic roots
no more exist and new techniques are needed to solve the questions.
Our approach relies on the tools borrowed from the theory of random
matrices.
This theory is closely connected to modern high-dimensional
statistics, and has provided in recent years
many efficient estimation and testing procedures
for  high-dimensional data analysis.
Excellent introduction and surveys on  this approach can be found in
\cite{Bai05}, \cite{Johnstone07}, \cite{JohnTitt09}
and \cite{PaulAue14}.
A  methodology particularly successful  both in theory and applications within this
approach relies on the
{\em spiked population model}  coined  in
\cite{Johnstone01}. This model  deals with one population only with a unit
population covariance matrix $I_p$ and the hypotheses are
simply $H_0: \Sigma_1=I_p$ versus $H_1: \Sigma_1=I_p+\Delta$ where $\Delta$ is a
rank-$M$ difference as in \eqref{eq:H1}. Again for small rank $M$, the
discrimination between both hypotheses will rely on the extreme
eigenvalues of the sample covariance matrix $S_1$.
Important results have been obtained in the last decade on the
behaviour of these extreme eigenvalues.
For example,
the  fluctuation of largest  eigenvalues of a sample covariance matrix
from  a complex spiked Gaussian population
is studied in \cite{BBP05}. These authors uncover a phase transition phenomenon: the
weak limit and the scaling of these extreme eigenvalues
are  different
depending on whether the  eigenvalues of $\Delta$ ({\em spikes})
are above, equal or below a critical value, situations refereed as
{\em super-critical, critical} and {\em sub-critical}, respectively.
In \cite{BaikSilv06}, the authors consider the spiked
population model with general populations
(not necessarily Gaussian). For  the almost sure limits of the
extreme sample eigenvalues of $S_1$, they find that
if a population spike (in $\Delta$) is  large or small enough,
the
corresponding sample spike eigenvalues will converge to a limit outside the
support of the limiting  spectrum ({\em outliers}).
In  \cite{Paul07}, a CLT is established
for these outliers, i.e.  the super-critical case,
under the Gaussian assumption
and assuming that population spikes are simple
(multiplicity 1).
The CLT for super-critical outliers with general populations
and arbitrary multiplicity numbers is developed  in
\cite{BaiYao08}. This theory has been later extended for
{\em generalised spiked population model} in
\cite{BaiYao12}.

In summary, from  the perspective of spiked population model, the
Fisher matrix $F=S_2^{-1}S_1$ under the alternative \eqref{eq:H1} can
be viewed as a {\em  spiked Fisher matrix}
and it is important to  establish a theory for  this 
two-population Fisher matrix $F$ in the vein of  
the  results discussed above  on the 
one-population spiked covariance matrix $S_1$.
As said before,  in \cite{DhaJoh14}, the
authors have already identified the transition phenomenon
for the extreme eigenvalues under the Gaussian assumption,
and these eigenvalues  are proved to be
asymptotic normal  assuming that the  spike eigenvalues in $\Delta$ are simple.
The main contributions of the paper are the following.
We prove that this phase transition phenomenon for extreme eigenvalues of a
spiked Fisher matrix is {\em universal}, valid for general populations under
some suitable moment conditions.
Next, we provide a general CLT  for the extreme sample
eigenvalues of $F$ in the super-critical regime: the limiting
distributions are not necessarily Gaussian; they are Gaussian 
{\em if and only if} the population spikes in $\Delta$ are simple.

In addition to the motivations  given so far on the importance of a
spiked Fisher matrix,  we are able  to implement an application of
the general theory developed in this paper
in the context of a signal detection problem with a large number of
detectors, see Section~\ref{sec:app}. Indeed, this problem has its own
interests and even with  quite limited experiments, we show that our
implementation can lead to very reliable solutions.

Finally, within the theory of random matrices, the techniques we use
in this paper for
spiked models are
 closely connected to other random matrix ensembles through the concept of small-rank
perturbations. The goal is again to examine the effect caused on
the extreme sample  eigenvalues by such perturbations.
Theories  on  perturbed  Wigner matrices can be found in
\cite{Peche06},
\cite{FerPec07},
\cite{CapDonFer09},
\cite{Pizzo13} and
\cite{RenSosh13}.
In a more general setting of finite-rank perturbation
including both the additive and the multiplicative one,
point-wisely
convergence of extreme eigenvalues is established in
\cite{BenNad11}
while their fluctuations are studied in  \cite{BGM11}.
In addition, \cite{BenNad11} contain also
results on spiked eigenvectors.

The rest of the paper is organised as  follows.
First,  the exact setting of the spiked Fisher matrix $F=S_2^{-1}S_1$ is
introduced in Section~\ref{sec:model}.
Then in  Section~\ref{sec:outlier},
we establish the phase transition
phenomenon for the extreme eigenvalues of $F$ where the transition
boundary is explicitly obtained.
Next, CLTs for those extreme eigenvalues  fluctuating around
some outliers (i.e. the super-critical case) are established first in Section~\ref{sec:CLT}
for one group of sample eigenvalues
corresponding to a  same population spike, and then in
Section~\ref{sec:CLT2} for all the groups jointly.
Section~\ref{ni} contains numerical illustrations that demonstrate  the finite sample performance of our results.
In Section~\ref{sec:app},  we develop in details a  signal detection
technique with prewhitening.
Proofs of the main theorems are included in these sections while some
technical lemmas are postponed into the Appendix \ref{sec:lemmas}.

\section{Spiked Fisher matrix and preliminary results}
\label{sec:model}

In what follows, we will assume that $\Sigma_2=I_p$. This assumption
does not loss any  generality since the eigenvalues of the Fisher
matrix $F=S_2^{-1}S_1$ are invariant  under the transformation
$S_1\mapsto \Sigma_2^{-1/2} S_1\Sigma_2^{-1/2}$,
$S_2\mapsto \Sigma_2^{-1/2} S_2\Sigma_2^{-1/2}$.
Also we will write $\Sigma_p$ for $\Sigma_1$ to signify the dependence
on the dimension $p$.
Therefore, 
the sample covariance matrices $S_1$  and $S_2$ that make up the
Fisher matrix $F=S_2^{-1}S_1$ are assumed to have the following structure.
Let
\begin{align}\label{z}
 Z = (z_1, \ldots z_n) =(z_{ij})_{1\le i\le p, 1\le j\le n}
\end{align}
and
\begin{align}\label{wkl}
 W = (w_1, \ldots w_T) =(w_{kl})_{1\le k\le p, 1\le l\le T}
\end{align}
be two independent arrays,  with respective size $p\times n$ and $p\times T$,
of independent real-valued random variables with mean 0 and variance
1.  The sample covariance matrix $S_2$ is
\begin{equation}
  \label{eq:S2}
  S_2=\frac1n \sum_{j=1}^n z_j z_j^* =
  \frac1n Z Z^* .
\end{equation}
Next,  $\Sigma_p$ is a  rank $M$ perturbation of $I_p$; therefore, we
can assume that it has the  spiked structure of form
\begin{equation}
  \label{eq:Sigma}
  \Sigma_p=\left(\begin{array}{cc}
      \Omega_M& 0\\
      0& I_{p-M}
    \end{array}\right),
\end{equation}
where   $ \Omega_M$ is a  $M\times M$ covariance matrix, $M$ being a
fixed constant,  containing $k$ spike eigenvalues
$(a_i)$,
$\displaystyle (\underbrace {a_1,\cdots, a_1}_{n_1},\cdots,\underbrace{ a_k,\cdots,
  a_k}_{n_k})$,
of respective multiplicity numbers $(n_i)$   ($n_1+\cdots
+n_k=M$). That is, $\Omega_M=U\diag({a_1,\cdots, a_1},\cdots,{ a_k,\cdots,
  a_k})U^*$, where $U$ is a $M \times M$ orthogonal matrix. Consider a sample
$x_1, \cdots, x_T$ of size $T$ that can be expressed as
$x_l:=\Sigma_p^{1/2}w_l$ and let
$X=(x_1,\ldots,X_T)=\Sigma_p^{1/2}W$.  The sample covariance matrix $S_1$
is
\begin{equation}
  \label{eq:S1}
  S_1=\frac1T \sum_{l=1}^T x_l x_l^* =\frac1T XX^*= \Sigma_p^{1/2}
  \left(\frac1T WW^*\right)\Sigma_p^{1/2}.
\end{equation}

Throughout the paper, we consider an asymptotic regime of
Mar\v{c}enko-Pastur type, i.e. 
\begin{equation}
  \label{eq:MPregime}
  p\wedge n \wedge T \rightarrow \infty, \quad
  y_p:=p/n \rightarrow y \in (0,1),\quad
  \text{and}  ~~ c_p:=p/T \rightarrow c >0.
\end{equation}

Recall that the {\em empirical spectral distribution} (ESD) of a
$p\times p$ matrix $A$  with eigenvalues $\{\lambda_j\}$ is the distribution $p^{-1}\sum_{j=1}^p\delta_{\lambda_j}$
where $\delta_a$ denotes the Dirac mass at $a$.  
Since the total rank $M$ generated by the $k$ spikes
is fixed, the ESD  of $F$ will have the same limit (LSD) 
as there were no spikes. This limiting spectral
distribution, the celebrated Wachter distribution,
has been known for a long time.

\begin{proposition}\label{LSD}
  For the Fisher matrix $F=S_2^{-1}S_1$ with the sample covariance
  matrices $S_i$'s given in \eqref{eq:S2}-\eqref{eq:S1}, assume that
  the dimension $p$ and the two sample sizes $n,T$ grow to infinity
  proportionally as in \eqref{eq:MPregime}. Then almost surely,
  the   ESD of $F$   weakly converges  to a deterministic
  distribution  $F_{c,y}$ with a bounded support $[b_1,b]$ and  a density function given by
  \begin{align}\label{fcy}
    f_{c,y}(x)=\left\{\begin{array}{ll}
        \frac{(1-y)\sqrt{(b-x)(x-b_1)}}{2\pi x (c+xy)}~,&
        \quad\text{when}~ b_1\le x\le b~,\\
        0~,& \quad \text{otherwise}~,
      \end{array}\right.
  \end{align}
  where
  \begin{align}\label{b}
    b_1=\left(\frac{1-\sqrt{c+y-cy}}{1-y}\right)^2\quad\text{and}\quad b=\left(\frac{1+\sqrt{c+y-cy}}{1-y}\right)^2.
  \end{align}
  Furthermore, if $c>1$, then $F_{c,y}$ has a point mass $1-1/c$ at the origin. Also, the Stieltjes transform $s(z)$ of $F_{c,y}$ equals:
  \begin{align}\label{st}
    s(z)=\frac{1}{zc}-\frac 1z-\frac{c(z(1-y)+1-c)+2zy-c\sqrt{(1-c+z(1-y))^2-4z}}{2zc(c+zy)}~,\quad z\notin [b_1,b].
  \end{align}
\end{proposition}

\begin{remark} Assuming both populations are Gaussian, 
  \citep[Theorem 3.1]{Wachter80} derives  the limiting distribution for roots of the
  determinental equation ,
  \[ |TS_1 -   x^2(TS_1+nS_2)|=0, \quad x\in \R.
  \]
  The continuous component of the distribution has a compact support
  $[A^2,B^2]$ with density function proportional to
  $\{(x-A^2)(B^2-x)\}^{1/2} / \{x(1-x^2)\}$. It can be readily checked
  that by the change of variable $z=c x^2/\{y(1-x^2)\}$, the density
  of the  continuous component of the LSD of  $F$   is exactly
  \eqref{fcy}. The validity of this limit for general populations (non
  necessarily Gaussian) is due to 
  \citet{Silverstein85} and \citet{BaiYinKrish87}.
\end{remark}

 For a complex number   $z\notin [b_1,b]$,  we define the following integrals  with respect to $F_{c,y}(x)$:
\begin{align}\label{smm}
  &s(z):=\int \frac{1}{x-z}dF_{c,y}(x)~,~~~~~~~
  m_1(z):=\int \frac{1}{(z-x)^2}dF_{c,y}(x)~,\nonumber\\
  &m_2(z):=\int \frac{x}{z-x}dF_{c,y}(x)~,~~\quad
  m_3(z):=\int \frac{x}{(z-x)^2}dF_{c,y}(x)~,\nonumber\\
  &m_4(z):=\int \frac{x^2}{(z-x)^2}dF_{c,y}(x)~.
\end{align}

\section{Phase transition of the extreme eigenvalues of $F=S_2^{-1}S_1$}
\label{sec:outlier}

In this section, we establish
a phase transition phenomenon for the extreme eigenvalues of $F=S_2^{-1}S_1$,
that is, when a population spike $a_i$ with multiplicity $n_i$
is larger (or smaller) than a critical value,   a packet of $n_i$  corresponding sample
eigenvalues  of
$F$ will  jump outside the support $[b_1,b]$ of its LSD $F_{c,y}$
and  converge all to a fixed limit.
Otherwise, these associated sample eigenvalues
will converge to one of  the edges  $b_1$ and  $b$.

For notation convenience, let $\ga=1/(1-y)\in(1,\infty)$. Define  the function
\begin{equation}
  \label{eq:phi}
  \phi(x) = \frac{\ga x(x-1+c)}{x-\ga},\quad x\ne \ga,
\end{equation}
which is a rational function  with a single  pole $\ga$.
An example is depicted in Figure~\ref{fig:phi} with parameters
$(c,y)=(\frac15,\frac12)$. The function has an asymptote of equation
$g(x)=\ga ( x+c-1+\ga)$ when $|x|\to\infty$.
\begin{figure}[htb]
  \centering
  \vspace{1cm}
  \includegraphics[width=0.7\linewidth]{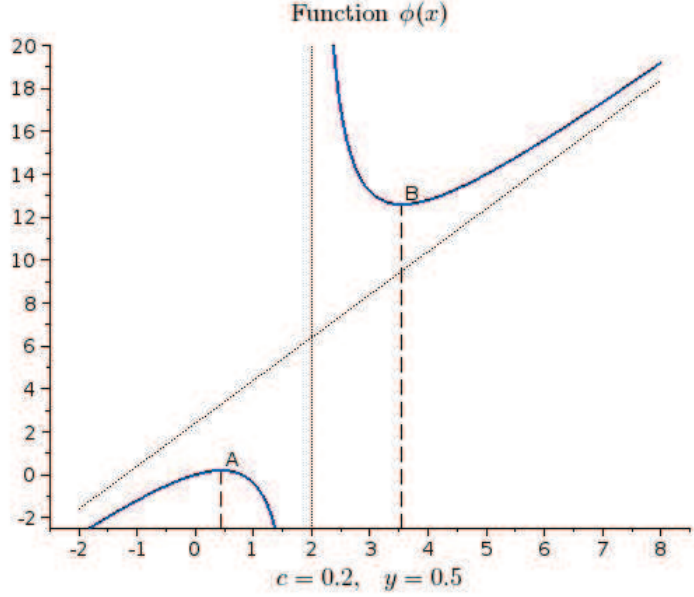}
  \caption{Example of the $\phi$ function with
    $(c,y)=(\frac15,\frac12)$
    and pole $\ga=2$. The asymptote has equation
    $y=2x+\frac{12}{5}$.
    The boundary points are $A(0.450,0.203)$ and $B(3.549,12.597)$
    meaning that critical values for spikes are 0.450 and 3.549 while
    the support of the LSD is [0.203,12.597].
  \label{fig:phi}}
\end{figure}

By assumption, the
$k$ population spike eigenvalues $\{a_i\}$  are all positive  and non unit.
We order them with their multiplicities
in  descending order together with the $p-M$ unit
eigenvalues as
\begin{align}
  a_1=\cdots=a_1>a_2=\cdots=a_2>\cdots
  >a_{k_0}=\cdots=a_{k_0}>1=\cdots=1 > \nonumber\\
  \quad  a_{k_0+1}=\cdots=a_{k_0+1}>\cdots>a_k=\cdots=a_k.
  \label{eq:ai}
\end{align}
That is, $k_0$ of these spike eigenvalues are larger than 1
while the other $k-k_0$ are smaller.
Let
\begin{align*}
  J_i=
  \begin{cases}
    [n_1+\cdots +n_{i-1}+1,n_1+\cdots +n_{i}]~,& \quad 1\leq i \leq k_0~,\\
    [p-(n_i+\cdots+n_k)+1, , p-(n_{i+1}+\cdots +n_{k})]~,& \quad k_0< i
    \leq k~.
  \end{cases}
\end{align*}
Notice that the cardinality of each $J_i$ is  $n_i$.
Next, the sample eigenvalues $\{l_{p,j}\}$ of the Fisher matrix $S_2^{-1}S_1$ are
also sorted in the descending order as
$l_{p,1}\ge l_{p,2}\ge \cdots \ge l_{p,p}$.
Therefore, for each spike eigenvalue $a_i$, there are
$n_i$ associated sample eigenvalues  $\{l_{p,j}, ~j\in J_i\}$.

\begin{theorem}\label{mainth1}
For the Fisher matrix $F=S_2^{-1}S_1$ with the sample covariance
  matrices $S_i$'s given in \eqref{eq:S2}-\eqref{eq:S1}, assume that
  the dimension $p$ and the two sample sizes $n,T$ grow to infinity
  proportionally as in \eqref{eq:MPregime}.
 Then for any spike eigenvalue $a_i$ ($i=1,
  \cdots, k$), it holds that
  for all $j\in J_i$,
  $l_{p,j}$ almost surely converges to a limit
  \begin{align}\label{limit}
    \la_i =
    \begin{cases}
      \phi(a_i), & \quad  |a_i-\ga|  > \ga \sqrt{c+y-cy} ~, \\
      b,   &  \quad     1< a_i \le \ga \{1+\sqrt{c+y-cy}\}~,\\
      b_1,  &  \quad     \ga \{1-\sqrt{c+y-cy}\} \le  a_i <1~.
    \end{cases}
  \end{align}
\end{theorem}

Basically, the theorem establishes a phase transition phenomenon for
the largest and smallest sample eigenvalues of a Fisher matrix.
Consider again the example  shown in Figure~\ref{fig:phi}.
The transition boundary is indicated with the boundary points $A$ and
$B$ with respective coordinates
\[  A(\ga \{1-\sqrt{c+y-cy}\},b_1)\quad \text{and} \quad B(\ga \{1+\sqrt{c+y-cy}\},b).
\]
When the spike is large enough or small enough, the corresponding
sample eigenvalues converge to $\phi(a_i)$ located outside the support
$[b_1,b]$ of the LSD of $F$. Otherwise, they converge to one of its  edges
$b_1$ and $b$.

It is worth observing that  when $y\to 0$, the $\phi(x)$ function
tends to  the function well-known in the literature for  similar transition phenomenon
of a spiked sample covariance matrix, i.e.
\begin{align}\label{phix}
 \lim_{y\to 0}\phi(x) =  x + \frac{  c x}{x-1},\quad x\ne 1,
\end{align}
see e.g. the $\psi$-function on Figure~4 of \cite{BaiYao12}.
These functions share a same shape; however the pole  here equals
$\ga=1/(1-y)$ which is larger than the pole 1  for the case of a
spiked sample covariance matrix.

As said in Introduction, this transition phenomenon has already been
established in a preprint  \cite{DhaJoh14} (their Proposition 5)
under Gaussian assumption and using a completely different approach.
Theorem~\ref{mainth1} proves that
such a phase transition phenomenon is indeed universal.

\begin{proof}(of Theorem \ref{mainth1})
The proof is divided into the following three steps:
\begin{itemize}
  \item Step 1: we derive the almost sure limit of an outlier eigenvalue of $S_2^{-1}S_1$;
  \item Step 2: we show that in order for the extreme eigenvalue of $S_2^{-1}S_1$ to be an outlier, the population spike $a_i$ should be larger (or smaller) than a critical value;
  \item Step 3: if not so, the extreme eigenvalue of $S_2^{-1}S_1$ will converge to one of the edge points $b$ and  $b_1$.
\end{itemize}

\noindent {\bf Step 1:}
Let $l_{p,j}~ (j\in J_i)$ be the outlier eigenvalue of $S_2^{-1}S_1$ corresponding to the population spike $a_i$. Then $l_{p,j}$  must satisfy the following equation:
\begin{align*}
|l_{p,j} I_p-S_2^{-1}S_1|=0~,
\end{align*}
and it is equivalent to
\begin{align}\label{e1}
|l_{p,j} S_2-S_1|=0~.
\end{align}
Now we make some short-hands. Denote $Z=\left(\begin{array}{c}
Z_1\\
Z_2\end{array}\right)$, where $Z_1$ is the $n$ observations of its
first $M$ coordinates  and $Z_2$ the remaining. We partition $X$
accordingly as  $X=\left(\begin{array}{c}
X_1\\
X_2\end{array}\right),$ where $X_1$ is  the $T$ observations of its first $M$ coordinates  and $X_2$ the remaining. Using such a representation, we have
\begin{align}\label{s12}
  S_1=\frac 1T XX^*=\frac 1T\left(\begin{array}{cc}
      X_1X^*_1& X_1X^*_2\\[2mm]
      X_2X^*_1& X_2X^*_2\end{array}\right),\quad 
  S_2=\frac 1n ZZ^*=\frac1n \left(\begin{array}{cc}
      Z_1Z^*_1& Z_1Z^*_2\\[2mm]
Z_2Z^*_1& Z_2Z^*_2\end{array}\right)~.
\end{align}

\noindent Then, \eqref{e1} could be written in the block form:
\begin{align}\label{e2}
\left|\left(\begin{array}{cc}
\frac {l_{p,j}}n Z_1Z^*_1-\frac 1T X_1X^*_1 & \frac {l_{p,j}}n Z_1Z^*_2-\frac 1T X_1X^*_2\\[2mm]
\frac {l_{p,j}}n Z_2Z^*_1-\frac 1T X_2X^*_1 & \frac {l_{p,j}}n Z_2Z^*_2-\frac 1T X_2X^*_2
\end{array}\right)\right|=0~.
\end{align}
Since $l_{p,j}$ is an outlier, it holds $|l_{p,j}\cdot\frac 1n Z_2Z^*_2-\frac 1T X_2X^*_2|\neq 0$, and
for block matrix, we have $\det\begin{pmatrix}
    A & B \\
    C & D \\
\end{pmatrix}=\det D \cdot \det(A-BD^{-1}C)$ when $D$ is invertible. Therefore, \eqref{e2} reduces to
\begin{align*}
&\left|\frac {l_{p,j}}{n} Z_1Z^*_1-\frac 1T X_1X^*_1\right.\\
&\quad-\left.\Big(\frac {l_{p,j}}{n} Z_1Z^*_2-\frac 1T X_1X^*_2\Big)\Big(\frac {l_{p,j}}{n} Z_2Z^*_2-\frac 1T X_2X^*_2\Big)^{-1}\Big(\frac {l_{p,j}}{n} Z_2Z^*_1-\frac 1T X_2X^*_1\Big)\right|=0~.
\end{align*}
More specifically, we have
\begin{align}\label{maine}
&\det \bigg(\underbrace{\frac {l_{p,j}}{n}Z_1\Big[I_n-Z^*_2\big(l_{p,j}I_p-(\frac 1n Z_2Z^*_2)^{-1}\frac 1T X_2X^*_2\big)^{-1}(\frac 1n Z_2Z^*_2)^{-1}\frac {l_{p,j}}{n} Z_2\Big]Z^*_1}_{(\uppercase\expandafter{\romannumeral1})}\nonumber\\
&\underbrace{-\frac 1TX_1\Big[I_T+X^*_2\big(l_{p,j}I_p-(\frac 1n Z_2Z^*_2)^{-1}\frac 1T X_2X^*_2\big)^{-1}(\frac 1n Z_2Z^*_2)^{-1}\frac 1TX_2\Big]X^*_1}_{(\uppercase\expandafter{\romannumeral2})}\nonumber\\
&+\underbrace{\frac {l_{p,j}}{n}Z_1Z^*_2\big(l_{p,j}I_p-(\frac 1n Z_2Z^*_2)^{-1}\frac 1T X_2X^*_2\big)^{-1}(\frac 1n Z_2Z^*_2)^{-1}\frac 1T X_2X^*_1}_{(\uppercase\expandafter{\romannumeral3})}\nonumber\\
&+\underbrace{\frac 1TX_1X^*_2\big(l_{p,j}I_p-(\frac 1n Z_2Z^*_2)^{-1}\frac 1T X_2X^*_2\big)^{-1}(\frac 1n Z_2Z^*_2)^{-1}\frac {l_{p,j}}{n}Z_2Z^*_1}_{(\uppercase\expandafter{\romannumeral4})}\bigg)\nonumber\\
&=0~.
\end{align}
In all the following, we denote by $S$  the Fisher matrix $\left(\frac
  1n Z_2Z^*_2\right)^{-1}\frac 1T X_2X^*_2$, which has a LSD $F_{c,y}(x)$. And in order to find the limit of $l_{p,j}$, we simply find the limit on the left hand side of \eqref{maine}, then it will generate an equation. Solving this equation will give the value of its limit.

\smallskip
\noindent First, consider the terms $(\uppercase\expandafter{\romannumeral3})$ and $(\uppercase\expandafter{\romannumeral4})$.
Since $(Z_1, X_1)$ is independent of $(Z_2, X_2)$, using Lemma \ref{f1}, we see these two terms will converge to some constant multiplied by  the covariance matrix between $X_1$ and $Z_1$. On the other hand,  $X_1$ is also independent of  $Z_1$, we have
\begin{align*}
\cov (X_1,Z_1)=\E X_1Z_1-\E X_1\E Z_1=\E X_1\E Z_1-\E X_1\E Z_1={\bf 0}_{M \times M}~.
\end{align*}
Therefore, these two terms will both tend to a zero matrix ${\bf 0}_{M
  \times M}$ almost surely.

So the remaining task is to find the limit of $(\uppercase\expandafter{\romannumeral1})$ and $(\uppercase\expandafter{\romannumeral2})$. We recall the expression of $X_1$ and $Z_1$  that
\begin{align*}
\cov (X_1)=U\diag(\underbrace{a_1,\cdots, a_1}_{n_1},\cdots, \underbrace {a_k, \cdots, a_k}_{n_k})U^{*}~,\quad \cov(Z_1)=I_M.
\end{align*}
According to Lemma \ref{f1}, we have
\begin{align}\label{r1}
(\uppercase\expandafter{\romannumeral1})&= \frac {l_{p,j}}{n}Z_1\bigg[I_n-Z^*_2(l_{p,j}I_p-S)^{-1}(\frac 1n Z_2Z^*_2)^{-1}\frac {l_{p,j}}{n} Z_2\bigg]Z^*_1\nonumber\\
&\rightarrow \frac {\lambda_i}{n}\left\{\E \tr \bigg[I_n-Z^*_2(\lambda_iI_p-S)^{-1}(\frac 1n Z_2Z^*_2)^{-1}\frac {\lambda_i}{n} Z_2\bigg]\right\}\cdot I_M\nonumber\\
&=\lambda_i(1+y\lambda_is(\lambda_i))\cdot I_M~,
\end{align}
here, we denote $\lambda_i$ as the limit of the outlier $\left\{l_{p,j}, j \in J_i\right\}$. For the same reason,
\begin{align}\label{r2}
(\uppercase\expandafter{\romannumeral2})&=-\frac 1TX_1\bigg[I_T+X^*_2(l_{p,j}I_p-S)^{-1}(\frac 1n Z_2Z^*_2)^{-1}\frac 1TX_2\bigg]X^*_1\nonumber\\
&\rightarrow -\frac 1T\left\{\e \tr \bigg[I_T+X^*_2(\lambda_iI_p-S)^{-1}(\frac 1n Z_2Z^*_2)^{-1}\frac 1TX_2\bigg]\right\}\cdot U \left(\begin{array}{c}
a_1 \qquad\\
~\ddots~\\
\qquad a_k
\end{array}\right)U^{*}\nonumber\\
&=U\big(-1+c+c\lambda_i s(\lambda_i)\big)\cdot \left(\begin{array}{c}
a_1 \qquad\\
~\ddots~\\
\qquad a_k
\end{array}\right)U^{*}~.
\end{align}

\noindent Therefore, combining \eqref{maine}, \eqref{r1} and \eqref{r2}, we have the determinant of the following $M \times M$ matrix
\begin{align*}
U\begin{pmatrix}
\lambda_i(1+y\lambda_i s(\lambda_i))+(-1+c+c\lambda_is(\lambda_i))a_1 \quad \quad\quad 0   \\
\vdots   \quad \quad \quad\quad \quad  \quad\quad \quad\ddots\quad\quad \quad\quad\quad \quad \quad\quad\quad\vdots  \\
\quad \quad \quad 0\quad \quad \quad \lambda_i(1+y\lambda_i s(\lambda_i))+(-1+c+c\lambda_is(\lambda_i))a_k &\\
\end{pmatrix}U^{*}
\end{align*}
equal to zero, which is also to say that $\lambda_i$ satisfies the equation:
\begin{align}\label{st1}
\lambda_i(1+y\lambda_i s(\lambda_i))+(-1+c+c\lambda_is(\lambda_i))a_i=0~.
\end{align}
Finally, together with the expression of the Stieltjes transform of a
Fisher matrix in \eqref{st}, we have
\begin{align*}
\lambda_i=\frac{a_i(a_i+c-1)}{a_i-a_iy-1}=\phi(a_i)~,
\end{align*}
where the function $\phi (x)$ is defined in \eqref{phix}.

\noindent {\bf Step 2:}
Define $\underline{s}(z)$ as the Stieltjes transform of the LSD of $\frac 1TX^*_2(\frac 1n Z_2Z^*_2)^{-1}X_2$, who shares the same non-zero eigenvalues as $S_2^{-1}S_1$. Then we have the relationship:
\begin{align}\label{est3}
\underline{s}(z)+\frac 1z (1-c)=cs(z)~.
\end{align}
Recall the expression of $s(z)$ in \eqref{st}, we have
\begin{align}\label{stss}
\underline{s}(z)=-\frac{c(z(1-y)+1-c)+2zy-c\sqrt{(1-c+z(1-y))^2-4z}}{2z(c+zy)}~.
\end{align}

\noindent On the other hand,  due to
\eqref{st1} and \eqref{est3}, we have the value for $\underline{s}(\lambda_i)$:
\begin{align}
    \underline{s}(\lambda_i)=\frac{yc-y-c}{y\lambda_i+a_ic}~.
\end{align}
Since $\lambda_i$ is outside the support of the LSD, we have
\begin{align*}
\underline{s}^{-1}\left(\frac{yc-y-c}{y\lambda_i+a_ic}\right)=\lambda_i>b\quad \text{or}\quad \underline{s}^{-1}\left(\frac{yc-y-c}{y\lambda_i+a_ic}\right)=\lambda_i<b_1~
\end{align*}
which is also to say that
\begin{align}\label{e5}
\underline{s}(b)<\frac{yc-y-c}{y\lambda_i+a_ic}~,
\end{align}
or
\begin{align}\label{sb1}
    \underline{s}(b_1)>\frac{yc-y-c}{y\lambda_i+a_ic}~.
\end{align}
Then \eqref{e5} says that $\underline{s}(b)$ must be smaller than the minimum value on its right hand side, whose minimum value is attained when $\lambda_i=b$ (the right hand side of \eqref{e5} is a decreasing function of $\lambda_i$).
Similarly, \eqref{sb1} says that $\underline{s}(b_1)$ must be larger than the maximum value on its right hand side, which is attained when $\lambda_i=b_1$.
Therefore,  the condition for  $\lambda_i$ be  an outlier is:
\begin{align}\label{e6}
\underline{s}(b)<\frac{yc-y-c}{yb+a_ic},\quad \text{or}\quad \underline{s}(b_1)>\frac{yc-y-c}{yb_1+a_ic}.
\end{align}
Finally, using \eqref{stss} together with the value of $b$ and $b_1$,
we have:
\begin{align*}
a_i>\frac{1+\sqrt{c+y-cy}}{1-y},~\quad \text{or}\quad a_i<\frac{1-\sqrt{c+y-cy}}{1-y},
\end{align*}
which is equivalent to say that (recall the expression of $\gamma$ that $\gamma=1/(1-y)$):
\begin{align*}
  |a_i-\gamma|>\gamma \sqrt{c+y-cy}~.
\end{align*}

\noindent{\bf Step 3:}
In this step, we show that if the condition in Step 2 is not
fulfilled, then the extreme eigenvalues of $S_2^{-1}S_1$ will tend to
one of the edge points $b_1$ and $b$. For simplicity, we only show the
convergence to the right edge $b$: the proof for the convergence to
the left edge $b_1$ is similar.  Thus suppose all the $a_i>1$ for $i=1, \cdots, k$.
For now, we make some short-hands. Let
\begin{align*}
S_1=\frac 1T XX^*=\frac1T\left(\begin{array}{cc}
 X_1X^*_1& X_1X^*_2\\[2mm]
 X_2X^*_1& X_2X^*_2\end{array}\right):=\begin{pmatrix}
B_{11} & B_{12} \\
B_{21} & B_{22} \\
\end{pmatrix}
\end{align*}
and
\begin{align*}
S_2=\frac 1n ZZ^*=\frac 1n\left(\begin{array}{cc}
 Z_1Z^*_1& Z_1Z^*_2\\[2mm]
 Z_2Z^*_1& Z_2Z^*_2\end{array}\right):=\begin{pmatrix}
A_{11} & A_{12} \\
A_{21} & A_{22} \\
\end{pmatrix}
~,
\end{align*}
where $B_{11}$ and $A_{11}$ are the corresponding blocks with size $M \times M$.
Using the inverse formula for block matrix,  the $(p-M)\times (p-M)$ major sub-matrix of $S^{-1}_2S_1$ is
\begin{align}\label{tar}
-(A_{22}-A_{21}A^{-1}_{11}A_{12})^{-1}A_{21}A^{-1}_{11}B_{12}+(A_{22}-A_{21}A^{-1}_{11}A_{12})^{-1}B_{22}:=C~.
\end{align}
The part
\begin{align*}
-(A_{22}-A_{21}A^{-1}_{11}A_{12})^{-1}A_{21}A^{-1}_{11}B_{12}=-(A_{22}-A_{21}A^{-1}_{11}A_{12})^{-1}A_{21}A^{-1}_{11}\cdot\frac 1T X_1X^*_2
\end{align*}
is of rank $M$; besides, we have
\begin{align*}
\tr\left\{(A_{22}-A_{21}A^{-1}_{11}A_{12})^{-1}A_{21}A^{-1}_{11}\frac 1T X_1X^*_2\right\}\rightarrow 0~,
\end{align*}
since $X_1$ is independent of $X_2$. Therefore, the $M$ nonzero eigenvalues of the matrix $-(A_{22}-A_{21}A^{-1}_{11}A_{12})^{-1}A_{21}A^{-1}_{11}B_{12}$ will all tend to zero (so is its largest one).
Then consider the second part of \eqref{tar} as follows.
\begin{align*}
A_{22}-A_{21}A^{-1}_{11}A_{12}=\frac 1n Z_2\left[I_n-Z^*_1\big(\frac 1n Z_1Z^*_1\big)^{-1}\frac 1n Z_1\right]Z^*_2:=\frac 1n Z_2 PZ_2~.
\end{align*}
Since  $P=I_n-Z^*_1\big(\frac 1n Z_1Z^*_1\big)^{-1}\frac 1n Z_1$ is  a projection matrix of rank $p-M$, it has the spectral decomposition:
\begin{align*}
P=V\left(\begin{array}{cccc}
0 &&&\\
&\ddots && \\
&& 0&\\
&&&{I}_{n-M} \\
\end{array}\right)V^*~,
\end{align*}
where $V$ is a $n \times n$ orthogonal matrix.
Since $M$ is fixed, the ESD of $P$ tends to $\delta_1$, which
leads to the fact that the LSD of the matrix $\frac 1n Z_2PZ^*_2$ is
 the standard Mar\v{c}enko-Pastur law. Then the matrix $(\frac 1n
Z_2PZ^*_2)^{-1}B_{22}$ is  a standard Fisher matrix, and its  largest
eigenvalues (finitely many) will tend to the right edge $b$ of the Wachter
distribution.  It follows then the two largest eigenvalues of $C$, say  
$\alpha_1(C)$ and $\alpha_2(C)$, also tend to $b$. 

Next  since $C$ is the $(p-M)\times(p-M)$ major sub-matrix of
$S^{-1}_2S_1$, we have by Cauchy  interlacing theorem
\[ \alpha_2(C)\le l_{p,M+1}\le\alpha_1(C)\le l_{p,1}~.
\] 
Thus $l_{p,M+1}\to b$ either. On the other hand, we have
\begin{align*}
  l_{p,1}= \| S^{-1}_2S_1\|_{op}\leq \|S_2^{-1}\|_{op} \cdot \|
  S_1\|_{op},
\end{align*}
so that for some positive constant $\theta$, 
$\limsup l_{p,1}\le \theta$. 
Consequently, almost surely, 
\[ b\le \liminf l_{p,M} \le \cdots \le \limsup l_{p,1}\le \theta<\infty~;
\] 
in particular the whole family $\{ l_{p,j}, ~1\le j\le M\}$ is
bounded. Now let $1\le j\le M$ be fixed and assume that
a   subsequence $(  l_{ p_k, j})_k$ converges to a limit $\beta\in[b,\theta]$. 
Either $\beta=\phi(a_i)>b$ or $\beta=b$. However, according to Step 2, 
$\beta > b$ implies that  $a_i>\gamma \{1+\sqrt{c+y-cy}\}$, and
otherwise,
we have $a_i\le \gamma \{1+\sqrt{c+y-cy}\}$.
Therefore, accordingly to one of these two conditions,
all subsequences converge to a {\em same}  limit $\phi(a_i)$ or $b$,
which is thus also the unique limit of the whole sequence 
$(l_{p,j})_p$.
\medskip
\noindent The proof of Theorem \ref{mainth1} is complete.
\end{proof}

\section{Central limit theorem for the outlier eigenvalues of
  $S^{-1}_2S_1$}
\label{sec:CLT}

The aim of this section is to give a CLT for the $n_i$-packed outlier eigenvalues:
\begin{align*}
  \sqrt p ~\{l_{p,j}-\phi(a_i), j \in J_i\}~.
\end{align*}
Denote $U=\begin{pmatrix}
            U_1 & U_2 & \cdots & U_k \\
          \end{pmatrix}~,
$
where each $U_i$ is a $M \times n_i$ matrix that corresponds to the $n_i$-packed spike eigenvalue $a_i$.

\begin{theorem}\label{mainth2}
  Assume the same assumptions as in Theorem \ref{mainth1} and in
  addition, the variables $(z_{ij})$ (in \eqref{z}) and $(w_{kl})$ (in
  \eqref{wkl}) have the same first four moments and denote $v_4$ as
  their common fourth moment:
  \[ v_4=\e |z_{ij}|^4=\e |w_{kl}|^4, \quad 1
  \leq i, k \leq p, ~ 1 \leq j \leq n, ~1 \leq l \leq T.\]
  Then for
  any population spike $a_i$ satisfying $|a_i-\gamma|>\gamma
  \sqrt{c+y-cy}$, the normalised $n_i$-packed outlier eigenvalues of
  $S_2^{-1}S_1$:
  $\sqrt p ~\{l_{p,j}-\phi(a_i), j \in J_i\}$ converge weakly to the distribution of the  eigenvalues of the random matrix
  $-U^{*}_iR(\lambda_i)U_i/\Delta(\lambda_i)$. Here,
  \begin{align}\label{del}
    \Delta(\lambda_i)=\frac{(1-a_i-c)(1+a_i(y-1))^2}{(a_i-1)(-1+2a_i+c+a_i^2(y-1))}~,
  \end{align}
  $R(\lambda_i)=(R_{mn})$
  is a $M \times M$ symmetric random matrix, made  with independent Gaussian entries of  mean zero and variance
  \begin{align}\label{varrij}
    \var (R_{mn})=\left\{\begin{array}{ll}
        2 \theta_i +(v_4-3)\omega_i~,& m=n~,\\
        \theta_i~, & m\neq n~,
      \end{array}\right.
  \end{align}
  where
  \begin{align}
    &\omega_i=\frac{a_i^2(a_i+c-1)^2(c+y)}{(a_i-1)^2}~,\\
    &\theta_i=\frac{a_i^2(a_i+c-1)^2(cy-c-y)}{-1+2a_i+c+a_i^2(y-1)}~.
  \end{align}
\end{theorem}

\noindent Numerical illustrations of this theorem are detailed in the next section.

\begin{remark}
Notice that the result above involves the $i$-th block $U_i$  of the
eigen-matrix $U$. When the spike $a_i$ is simple, $U_i$ is unique up
to its sign, then  $U^{*}_iR(\lambda_i)U_i$ is uniquely
determined. But when $a_i$ has multiplicities greater than 1, $U_i$ is
not unique; actually, any rotation of $U_i$ can be an eigenvector
matrix  corresponding to $a_i$. Therefore, Lemma \ref{unique} in the Appendix states that, such a rotation will not affect  the eigenvalues of the matrix $U^{*}_iR(\lambda_i)U_i$.
\end{remark}
\begin{proof}(proof of Theorem \ref{mainth2})

  \noindent{\bf Step 1: Convergence to the eigenvalues of the random matrix
    $-U^{*}_iR(\lambda_i)U_i/\Delta(\lambda_i)$.}
  We start from \eqref{maine}. First we make some short hands.
  Define
  \begin{align}\label{abcd}
    A(\lambda)&=I_n-Z^*_2\bigg[\lambda I_p-\Big(\frac 1n Z_2Z^*_2\Big)^{-1}\frac 1T X_2X^*_2\bigg]^{-1}\Big(\frac 1n Z_2Z^*_2\Big)^{-1}\frac{\lambda}{n}Z_2~,\nonumber\\
    B(\lambda)&=I_T+X^*_2\bigg[\lambda I_p-\Big(\frac 1n Z_2Z^*_2\Big)^{-1}\frac 1T X_2X^*_2\bigg]^{-1}\Big(\frac 1n Z_2Z^*_2\Big)^{-1}\frac{1}{T}X_2~,\nonumber\\
    C(\lambda)&=Z^*_2\bigg[\lambda I_p-\Big(\frac 1n Z_2Z^*_2\Big)^{-1}\frac 1T X_2X^*_2\bigg]^{-1}\Big(\frac 1n Z_2Z^*_2\Big)^{-1}\frac{1}{T}X_2~,\nonumber\\
    D(\lambda)&=X^*_2\bigg[\lambda I_p-\Big(\frac 1n Z_2Z^*_2\Big)^{-1}\frac 1T X_2X^*_2\bigg]^{-1}\Big(\frac 1n Z_2Z^*_2\Big)^{-1}\frac{1}{n}Z_2~,
  \end{align}
  then \eqref{maine} could be written as
  \begin{align}\label{m1}
    \det \bigg(\underbrace{\frac {l_{p,j}}{n} Z_1A(l_{p,j})Z^*_1}_{(\romannumeral1)}-\underbrace{\frac 1TX_1B(l_{p,j})X^*_1}_{(\romannumeral2)}+\underbrace{\frac {l_{p,j}}{n} Z_1C(l_{p,j})X^*_1}_{(\romannumeral3)}+\underbrace{\frac {l_{p,j}}{T} X_1D(l_{p,j})Z^*_1}_{(\romannumeral4)}\bigg)=0~.
  \end{align}
  The remaining is to find second order approximation of the four terms on the left hand side of \eqref{m1}.

  Using Lemma \ref{asl} in the appendix, we have
  \begin{align}\label{p1}
    (\romannumeral1)&=\e \frac{\lambda_i}{n} Z_1A(\lambda_i)Z^*_1+\frac {l_{p,j}}n Z_1A(l_{p,j})Z^*_1-\e \frac{\lambda_i}{n} Z_1A(\lambda_i)Z^*_1\nonumber\\
    &=(\lambda_i+y\lambda_i^2s(\lambda_i))\cdot I_M+ \frac{l_{p,j}}{n} Z_1A(l_{p,j})Z^*_1- \frac{\lambda_i}{n} Z_1A(\lambda_i)Z^*_1+\frac{\lambda_i}{n} Z_1A(\lambda_i)Z^*_1-\e \frac{\lambda_i}{n} Z_1A(\lambda_i)Z^*_1\nonumber\\
    &=(\lambda_i+y\lambda_i^2s(\lambda_i))\cdot I_M+ \frac{l_{p,j}-\lambda_i}{n} Z_1A(l_{p,j})Z^*_1+\frac{\lambda_i}{n} Z_1\left(A(l_{p,j})-A(\lambda_i)\right)Z^*_1\nonumber\\
    &\quad +\frac{\lambda_i}{\sqrt n}\Big[\frac{1}{\sqrt n}Z_1A(\lambda_i)Z^*_1-\e \frac{1}{\sqrt n}Z_1A(\lambda_i)Z^*_1\Big]\nonumber\\
    &\rightarrow (\lambda_i+y\lambda_i^2s(\lambda_i))\cdot I_M+(l_{p,j}-\lambda_i)\cdot(1+2y\lambda_i s(\lambda_i)+\lambda_i^2 y m_1(\lambda_i))\cdot I_M\nonumber\\
    &\quad +\frac{\lambda_i}{\sqrt n}\Big[\frac{1}{\sqrt n}Z_1A(\lambda_i)Z^*_1-\e \frac{1}{\sqrt n}Z_1A(\lambda_i)Z^*_1\Big]~,
  \end{align}

  \begin{align}\label{p2}
    (\romannumeral2)&=\e \frac 1TX_1B(\lambda_i)X^*_1+\frac 1TX_1B(l_{p,j})X^*_1-\e \frac 1TX_1B(\lambda_i)X^*_1\nonumber\\
    &=U\big(1-c-c\lambda_i s(\lambda_i)\big)\cdot \left(\begin{array}{c}
        a_1 \qquad   \\
        ~ \ddots ~ \\
        \qquad a_k \\
      \end{array}\right)U^{*}+\frac 1TX_1(B(l_{p,j})-B(\lambda_i))X^*_1\nonumber\\
    &\quad +\frac {1}{\sqrt T}\Big[\frac{1}{\sqrt T}X_1B(\lambda_i)X^*_1-\e \frac{1}{\sqrt T}X_1B(\lambda_i)X^*_1\Big]\nonumber\\
    &\rightarrow U\big(1-c-c\lambda_i s(\lambda_i)\big)\cdot \left(\begin{array}{c}
        a_1 \qquad   \\
        ~ \ddots ~  \\
        \qquad a_k \\
      \end{array}\right)U^{*}-U(l_{p,j}-\lambda_i)\cdot cm_3(\lambda_i)\cdot \left(\begin{array}{c}
        a_1 \qquad   \\
        ~ \ddots ~ \\
        \qquad a_k \\
      \end{array}\right)U^{*}\nonumber\\
    &\quad +\frac {1}{\sqrt T}\Big[\frac{1}{\sqrt T}X_1B(\lambda_i)X^*_1-\e \frac{1}{\sqrt T}X_1B(\lambda_i)X^*_1\Big]~,
  \end{align}

  \begin{align}\label{p3}
    (\romannumeral3)&=\frac {l_{p,j}}{n} Z_1C(l_{p,j})X^*_1-\e \frac {\lambda_i}{n} Z_1C(\lambda_i)X^*_1\nonumber\\
    &=\frac {l_{p,j}}{n} Z_1C(l_{p,j})X^*_1-\frac {\lambda_i}{n} Z_1C(\lambda_i)X^*_1+\frac {\lambda_i}{n} Z_1C(\lambda_i)X^*_1-\e \frac {\lambda_i}{n} Z_1C(\lambda_i)X^*_1\nonumber\\
    &=\frac {l_{p,j}}{n} Z_1(C(l_{p,j})-C(\lambda_i))X^*_1+\frac {l_{p,j}-\lambda_i}{n} Z_1C(\lambda_i)X^*_1+\frac{\lambda_i}{n}\cdot\Big[ Z_1C(\lambda_i)X^*_1-\e Z_1C(\lambda_i)X^*_1\Big]\nonumber\\
    &\rightarrow\frac{\lambda_i}{n}\cdot\Big[Z_1C(\lambda_i)X^*_1-\e Z_1C(\lambda_i)X^*_1\Big]~,
  \end{align}

  \begin{align}\label{p4}
    (\romannumeral4)&=\frac {l_{p,j}}{T} X_1D(l_{p,j})Z^*_1-\e \frac {\lambda_i}{T} X_1D(\lambda_i)Z^*_1\nonumber\\
    &=\frac {l_{p,j}}{T} X_1D(l_{p,j})Z^*_1-\frac {\lambda_i}{T} X_1D(\lambda_i)Z^*_1+\frac {\lambda_i}{T} X_1D(\lambda_i)Z^*_1-\e \frac {\lambda_i}{T} X_1D(\lambda_i)Z^*_1\nonumber\\
    &=\frac {l_{p,j}}{T} X_1(D(l_{p,j})-D(\lambda_i))Z^*_1+\frac {l_{p,j}-\lambda_i}{T} X_1D(\lambda_i)Z^*_1+\frac {\lambda_i}{T}\cdot \Big[X_1D(\lambda_i)Z^*_1-\e X_1D(\lambda_i)Z^*_1\Big]\nonumber\\
    &\rightarrow \frac{\lambda_i}{T}\cdot \Big[X_1D(\lambda_i)Z^*_1-\e X_1D(\lambda_i)Z^*_1\Big]~.
  \end{align}
  Denote
  \begin{align}
    R_n(\lambda_i)
    &= \lambda_i \sqrt{\frac p n} \bigg[\frac{1}{\sqrt n}Z_1A(\lambda_i)Z^*_1\bigg]-\sqrt{\frac p T} \bigg[\frac{1}{\sqrt T}X_1B(\lambda_i)X^*_1\bigg]
    +\lambda_i\sqrt{\frac p n}\bigg[\frac {1}{\sqrt n}Z_1C(\lambda_i)X^*_1\bigg]\nonumber\\
    &\quad + \lambda_i \sqrt{\frac p T}\bigg[\frac{1}{\sqrt T}X_1D(\lambda_i)Z^*_1\bigg]-\e [\cdot]~,\label{rnl}
  \end{align}
  where $\e[\cdot]$ denotes the total expectation of all the preceding
  terms in the equation,
  and
  \begin{align*}
    \Delta(\lambda_i)&=1+2y\lambda_i s(\lambda_i)+\lambda_i^2 y m_1(\lambda_i)+acm_3(\lambda_i)~.\\
  \end{align*}
  Combining \eqref{m1}, \eqref{p1}, \eqref{p2}, \eqref{p3}, \eqref{p4} and considering the diagonal block that corresponds to the row and column index in $J_i \times J_i$ leads to:
  \begin{align}\label{final}
    \Big|\sqrt p (l_{p,j}-\lambda_i)\cdot \Delta(\lambda_i)\cdot I_{n_i}+[U^{*}R_n(\lambda_i)U]_i\Big|\rightarrow 0~.
  \end{align}
  Furthermore,
  it will be established in Step 2 below that
  \begin{equation}
    \label{eq:covRn}
    [U^{*}R_n(\lambda_i)U]_i  \longrightarrow [U^{*}R(\lambda_i)U]_i  \quad \text{in distribution},
  \end{equation}
  for some random matrix $R(\lambda_i)$. Using the device of Skorokhod
  strong representation \citep{skoro56,HuBai14}, we may assume
  that this convergence hold almost surely by considering an enlarged
  probability space. Under this device,  \eqref{final} is equivalent to say that
  $
  \sqrt p (l_{p,j}-\lambda_i)
  $
  tends to an eigenvalue of the matrix $-[U^{*}R(\lambda_i)U]_i/\Delta(\lambda_i)(=-U^{*}_iR(\lambda_i)U_i/\Delta(\lambda_i))$. Finally, as the index $j$ is arbitrary over the set $J_i$, all the $n_i$ random variables
  \begin{align*}
    \left\{\sqrt p (l_{p,j}-\lambda_i), j \in J_i\right\}
  \end{align*}
  converge almost surely to the set of  eigenvalues of the random matrix $-U^{*}_iR(\lambda_i)U_i/\Delta(\lambda_i)$.
  Besides, due to Lemma \ref{five}, we have
  \begin{align*}
    \Delta(\lambda_i)&=1+2y\lambda_i s(\lambda_i)+\lambda_i^2 y m_1(\lambda_i)+acm_3(\lambda_i)\\
    &=\frac{(1-a_i-c)(1+a_i(y-1))^2}{(a_i-1)(-1+2a_i+c+a_i^2(y-1))}~.
  \end{align*}

  \noindent{\bf Step 2: Proof of the convergence \eqref{eq:covRn} and
    structure of the random matrix $R(\lambda_i)$.}
  In the second step, we aim to find the matrix limit of the block random matrix $[U^{*}R_n(\lambda_i)U]_i$. First, we show $[U^{*}R_n(\lambda_i)U]_i$  equals to another random matrix $[U^{*}\tilde{R}_n(\lambda_i)U]_i$, here $\tilde{R}_n(\lambda_i)$ is the type of random sesquilinear form. Then  using the results in \cite{BaiYao08} (Proposition 3.1 and Remark 1), we are able to find the matrix limit of $\tilde{R}_n(\lambda_i)$.

  \noindent By assumption (b) that $x_i=\Sigma^{1/2}_ps_i$, we have its first $M$ components
  \begin{align*}
    X_1=\Omega^{1/2}_pS_1=U\left(\begin{array}{c}
        \sqrt{a_1} \qquad\\
        ~\ddots~\\
        \qquad \sqrt{a_k}
      \end{array}\right)U^{*}S_1~.
  \end{align*}
  Recall the definition of $R_n(\lambda_i)$ in \eqref{rnl}, we have
  \begin{align}\label{ff}
    &\quad~ U^{*}R_n(\lambda_i)U\nonumber\\
    &=U^{*}\frac {\sqrt p \lambda_i}{ n}Z_{1}A(\lambda_i)Z^*_{1}U- \frac {\sqrt p}{T}\left(\begin{array}{c}
        \sqrt{a_1} \qquad\\
        ~\ddots~\\
        \qquad \sqrt{a_k}
      \end{array}\right)U^{*}S_1B(\lambda_i)S^*_{1}U\left(\begin{array}{c}
        \sqrt{a_1} \qquad\\
        ~\ddots~\\
        \qquad \sqrt{a_k}
      \end{array}\right)\nonumber\\
    &\quad+U^{*}\frac {\sqrt p \lambda_i}{n}Z_1C(\lambda_i)S^*_1U\left(\begin{array}{c}
        \sqrt{a_1} \qquad\\
        ~\ddots~\\
        \qquad \sqrt{a_k}
      \end{array}\right)+\frac{\lambda_i \sqrt p } { T}\left(\begin{array}{c}
        \sqrt{a_1} \qquad\\
        ~\ddots~\\
        \qquad \sqrt{a_k}
      \end{array}\right)U^{*}S_1D(\lambda_i)Z^*_1U\nonumber\\
    &\quad -\e [\cdot]~.
  \end{align}
  Therefore, if we consider its $i$-th block that corresponds to the row and column index in the set $J_i \times J_i$:
  \begin{align}
    &\quad ~ [U^{*}R_n(\lambda_i)U]_i\nonumber\\
    &= \lambda_i\sqrt{\frac p n}\Big[\frac {1}{\sqrt n}U^{*}Z_{1}A(\lambda_i)Z^*_{1}U\Big]_i-  a_i \sqrt{\frac p T} \Big[\frac {1}{\sqrt T}U^{*}S_1B(\lambda_i)S^*_{1}U\Big]_i\nonumber\\
    & \quad+\lambda_i\sqrt {a_i}\sqrt{\frac p n} \Big[\frac {1}{\sqrt n}U^{*}Z_1C(\lambda_i)S^*_1U\Big]_i+\lambda_i \sqrt {a_i}\sqrt{\frac p T}\Big[\frac{1} {\sqrt T}U^{*}S_1D(\lambda_i)Z^*_1U\Big]_i\nonumber\\
    &\quad -\e [\cdot]\nonumber\\
    &= \Big[\lambda_i\frac {\sqrt p}{n}U^{*}Z_{1}A(\lambda_i)Z^*_{1}U-  a_i\frac {\sqrt p}{ T}U^{*}S_1B(\lambda_i)S^*_{1}U\nonumber\\
    & \quad+\lambda_i\sqrt {a_i} \frac {\sqrt p}{ n}U^{*}Z_1C(\lambda_i)S^*_1U+\lambda_i \sqrt {a_i}\frac{\sqrt p} { T}U^{*}S_1D(\lambda_i)Z^*_1U\Big]_i\nonumber\\
    &\quad -\e [\cdot]\nonumber\\
    &=\left[U^{*}\begin{pmatrix}
        Z_1 & S_1 \\
      \end{pmatrix}\left(\begin{array}{cc}
          \frac{\lambda_i \sqrt p A(\lambda_i)}{ n} & \frac{\lambda_i  \sqrt {a_ip} C(\lambda_i)}{ n}\\[2mm]
          \frac{\lambda_i \sqrt {a_ip} D(\lambda_i)}{ T} & \frac{-a_i \sqrt p B(\lambda_i)}{T}
        \end{array}\right) \begin{pmatrix}
        Z^*_1 \\[3mm]
        S^*_1 \\
      \end{pmatrix}U-\E [\cdot]\right]_i\nonumber\\
    &:=[U^{*}\tilde{R}_n(\lambda_i)U]_i\nonumber\\
    &=U^{*}_i\tilde{R}_n(\lambda_i)U_i~,
  \end{align}
  where
  \begin{align*}
    \tilde{R}_n(\lambda_i):=\begin{pmatrix}
      Z_1 & S_1 \\
    \end{pmatrix}\left(\begin{array}{cc}
        \frac{\lambda_i \sqrt p A(\lambda_i)}{ n} & \frac{\lambda_i  \sqrt {a_ip} C(\lambda_i)}{ n}\\[2mm]
        \frac{\lambda_i \sqrt {a_ip} D(\lambda_i)}{ T} & \frac{-a_i \sqrt p B(\lambda_i)}{T}
      \end{array}\right) \begin{pmatrix}
      Z^*_1 \\[3mm]
      S^*_1 \\
    \end{pmatrix}-\E [\cdot]~.
  \end{align*}
  Finally, using Lemma \ref{lemm3} in the appendix leads to the result.
  The proof of Theorem \ref{mainth2} is complete.
\end{proof}

\noindent Next we consider a special case  where  $\Omega_p$ is diagonal, whose eigenvalues being all simple. In other words, we have $M=k$ and $n_i=1$ for all $1\leq i\leq M$. Hence $U=I_M$.
Following  Theorem \ref{mainth2}, we can derive the asymptotic normality  for the normalised outlier eigenvalues of $S^{-1}_2S_1$ when $|a_i-\gamma|>\gamma \sqrt{c+y-cy}$.
\begin{proposition}\label{sim}
  Under the same assumptions as in Theorem \ref{mainth1}, with additional conditions that
  $\Omega_p$ is diagonal and all its eigenvalues $a_i$ ($1\leq i \leq M$) are simple,
  we have when  $|a_i-\gamma|>\gamma \sqrt{c+y-cy}$, the outlier eigenvalue  $l_i$ of $S_2^{-1}S_1$ is asymptotically Gaussian:
  \begin{align*}
    \sqrt p ~\left(l_{i}-\frac{a_i(a_i-1+c)}{a_i-1-a_iy}\right)\Longrightarrow N(0, \sigma^2_i)~,
  \end{align*}
  where
  \begin{align*}
    \sigma^2_i&=\frac{2a^2_i(cy-c-y)(a_i-1)^2(-1+2a_i+c+a^2_i(y-1))}{(1+a_i(y-1))^4}\\
    &\quad+(v_4-3)\cdot\frac{a^2_i(c+y)(-1+2a_i+c+a^2_i(y-1))^2}{(1+a_i(y-1))^4}~.
  \end{align*}
\end{proposition}

\begin{remark}
  Notice that when the data are standard Gaussian, we have $v_4=3$, then the above theorem reduces to
  \begin{align*}
    &\sqrt p \left(l_{i}-\frac{a_i(a_i-1+c)}{a_i-1-a_iy}\right)\\
    &\Longrightarrow N\left(0,\frac{2a_i^2(a_i-1)^2(cy-c-y)(-1+2a_i+c+a_i^2(y-1))}{(1+a_i(y-1))^4}\right)~,
  \end{align*}
  which is exactly the result in \cite{DhaJoh14}, see setting 1 in their Proposition 11.
\end{remark}

\begin{proof}(of Proposition \ref{sim})\quad  Under the above assumptions, the random matrix $-[U^{*}R(\lambda_i)U]_i$ reduces to $-[R(\lambda_i)]_i$. And since all the $n_i=1$, we have $-[R(\lambda_i)]_i$ equals the $(i,i)$-th element of $-R(\lambda_i)$,  which is a  Gaussian random variable with  mean zero and variance
  \begin{align*}
    2 \theta_i +(v_4-3)\omega_i=\frac{2a_i^2(a_i+c-1)^2(cy-c-y)}{-1+2a_i+c+a_i^2(y-1)} +(v_4-3)\cdot \frac{a_i^2(a_i+c-1)^2(c+y)}{(a_i-1)^2}.
  \end{align*}
  Therefore, combining with \eqref{del} we have
  \begin{align*}
    \sqrt p \left(l_{i}-\frac{a_i(a_i-1+c)}{a_i-1-a_iy}\right) \Longrightarrow N(0,\sigma^2_i)~,
  \end{align*}
  where
  \begin{align*}
    \sigma^2_i&=\frac{2a^2_i(cy-c-y)(a_i-1)^2(-1+2a_i+c+a^2_i(y-1))}{(1+a_i(y-1))^4}\\
    &\quad+(v_4-3)\cdot\frac{a^2_i(c+y)(-1+2a_i+c+a^2_i(y-1))^2}{(1+a_i(y-1))^4}~.
  \end{align*}

  \noindent The proof of Proposition \ref{sim} is complete.

\end{proof}



\section{Numerical illustrations}\label{ni}
In this section, numerical results are provided to illustrate the results of our Theorem \ref{mainth2} and Proposition \ref{sim}.
 We fix  $p=200$, $T=1000$, $n=400$ with 1000 replications, thus $y=1/2$ and $c=1/5$. The critical interval is then $[\gamma-\gamma\sqrt{c+y-cy}, \gamma+\gamma\sqrt{c+y-cy}]=[0.45,3.55]$ and the limiting support $[b_1, b]=[0.2, 12.6]$. Consider $k=3$ spike eigenvalues $(a_1, a_2, a_3)=(20, 0.2, 0.1)$ with respective multiplicity $(n_1, n_2, n_3)=(1, 2, 1)$. Let $l_1\geq \cdots \geq l_p$ be the ordered eigenvalues of the Fisher matrix $S_2^{-1}S_1$. We are particularly interested in the distributions of  $l_1$, $(l_{p-2}, l_{p-1})$ and $l_p$, which  corresponds to the spike eigenvalues $a_1$, $a_2$ and $a_3$, respectively.

\subsection{Case of $U=I_4$}\label{ui}
In this subsection, we consider a simple case that $U=I_4$. Therefore, following Theorem \ref{mainth2}, we have
\begin{itemize}
  \item for $j=1, p$, $\sqrt p \{l_j-\phi (a_i)\}\rightarrow N(0, \sigma^2_i)$.
        Here, for $j=1$, $i=1$, $\phi (a_1)=42.67$ and $\sigma^2_1=4246.8+1103.5(v_4-3)$; and for $j=p$, $i=3$, $\phi (a_3)=0.07$ and $\sigma^2_3=7.2\times10^{-3}+3.15\times 10^{-3}(v_4-3)$.
  \item for $j=p-2, p-1$ and $i=2$, the two dimensional random vector $\sqrt p \{l_j-\phi (a_2)\}$ converges to the eigenvalues of the random matrix $-\frac{R_{mn}}{\Delta(\lambda_2)}$. Here, $\phi (a_2)=0.13$, $\Delta(\lambda_2)=1.45$ and $R_{mn}$ is the $2 \times 2$ symmetric random matrix, made with independent Gaussian entries of mean zero and variance
      \begin{align}\label{rmn}
         \var (R_{mn})=\left\{\begin{array}{lll}
        2 \theta_2 +(v_4-3)\omega_2 &(=0.04+0.016(v_4-3))~,& m=n~,\\
        \theta_2 &(=0.02)~, & m\neq n~,
      \end{array}\right.
  \end{align}
\end{itemize}
Simulations are conducted to compare the distributions of the empirical extreme eigenvalues with their limits.

\subsubsection{Gaussian case}

\begin{figure}[htbp!]
  \centering
  \vspace{1cm}
  \includegraphics[width=0.9\linewidth,height=5cm]{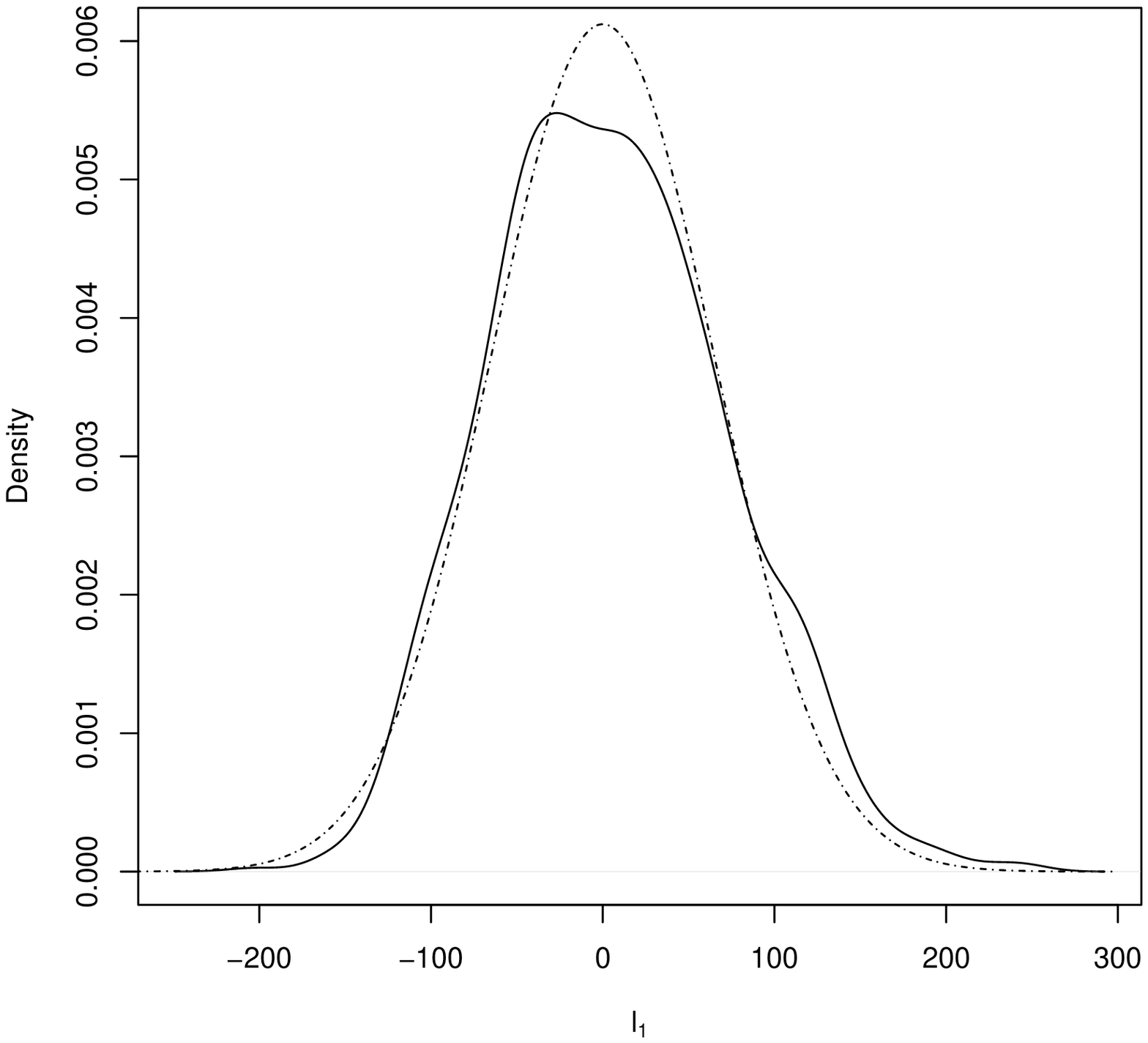}

  \includegraphics[width=0.9\linewidth,height=5cm]{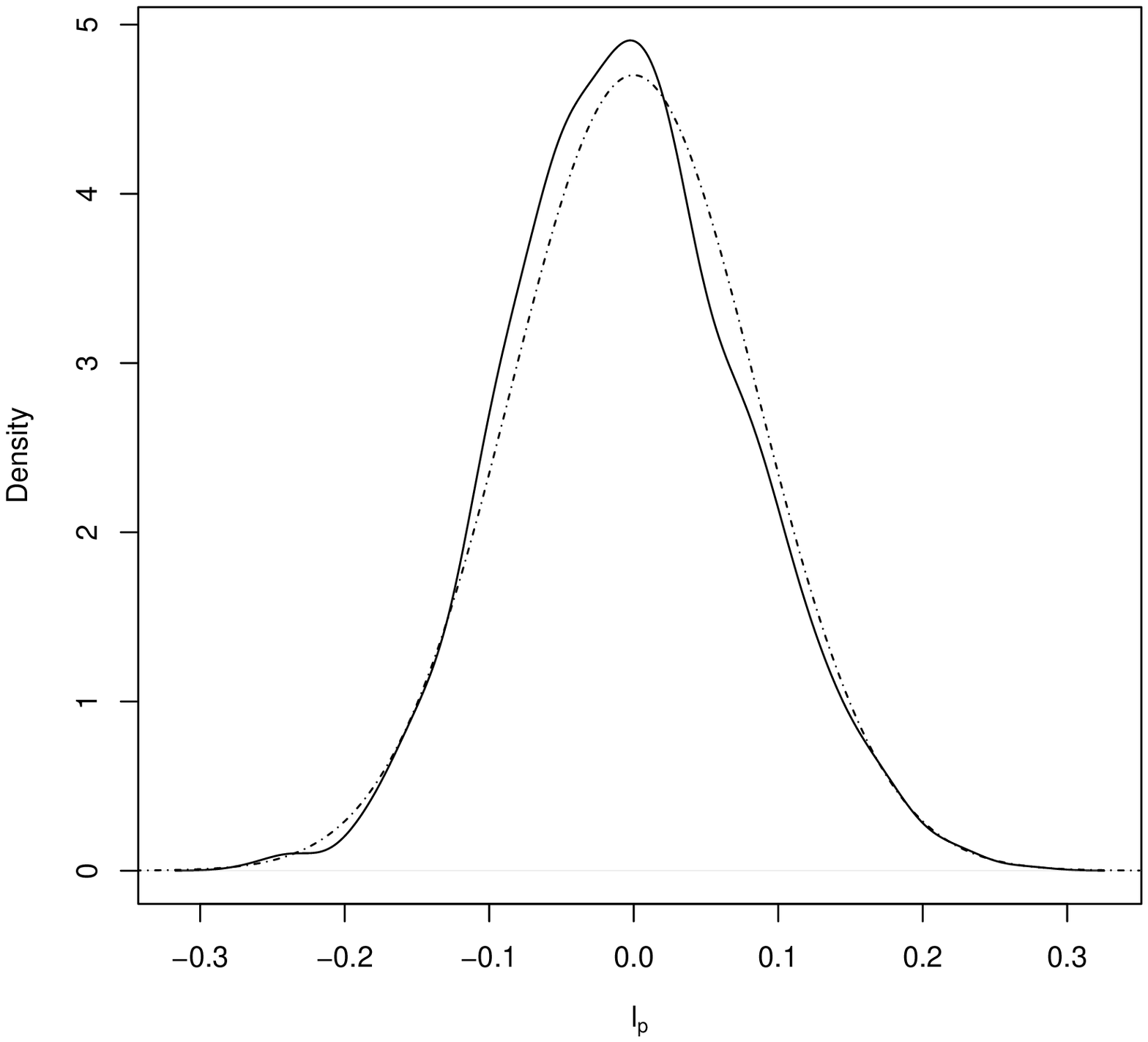}

  \mbox{
  \includegraphics[width=0.5\linewidth]{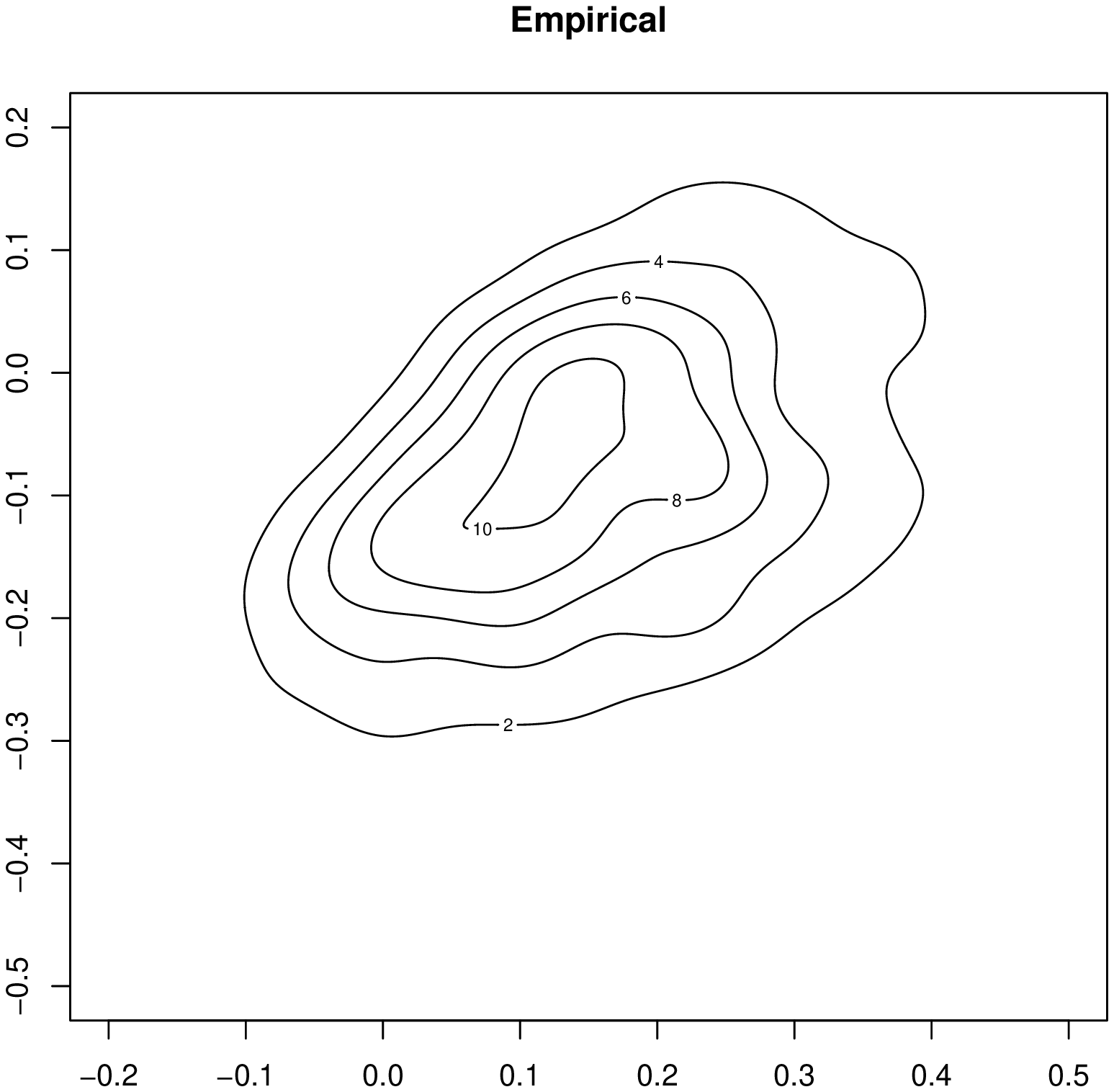}
  \includegraphics[width=0.5\linewidth]{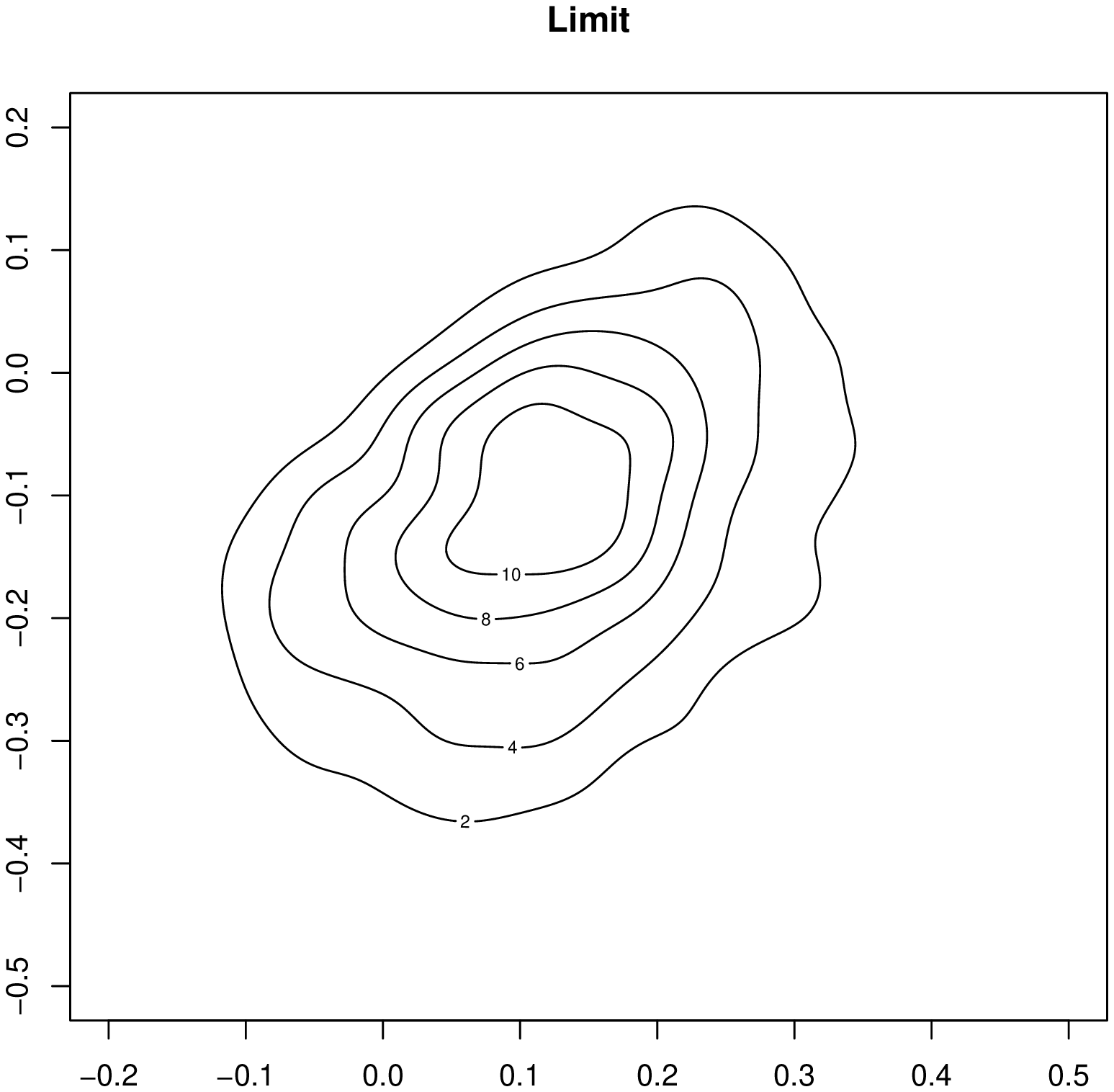}
  }

  \caption{Upper panels show the empirical densities  of $l_1$ and
    $l_p$ (solid lines, after centralisation and scaling)
    compared to their Gaussian limits (dashed lines).
    Lower panels show contour plots of  empirical joint  density
    function  of $(l_{p-2}, l_{p-1})$ (left plot, after centralisation and scaling)
    and contour plots of their limits  (right plot). Both the empirical and limit joint density functions are displayed using the two-dimensional kernel density estimates.
    Samples are from i.i.d. standard Gaussian distribution with
    $U=I_4$ with
    1000 independent replications.
    \label{gaussianextreme}}
\end{figure}

First, we assume all the  $z_{ij}$ and $w_{ij}$ are i.i.d. standard Gaussian, thus $v_4-3=0$. And according to \eqref{rmn}, $R_{mn}/\sqrt{0.04}$ is the standard $2 \times 2$ Gaussian Wigner matrix (GOE). Therefore, we have
 \begin{itemize}
   \item $\sqrt p \{l_1-42.67\}\rightarrow N(0, 4246.8)$~,
   \item $\sqrt p \{l_p-0.07\}\rightarrow N(0, 7.2\times10^{-3})$~,
   \item The two-dimensional random vector
     $  \sqrt p \{l_{p-2}-0.13,l_{p-1}-0.13  \} $
     converges to the eigenvalues of the random matrix $-0.138\cdot W$, here $W$ is a $2 \times 2$ GOE.
 \end{itemize}

 \noindent Figure \ref{gaussianextreme}, upper panels,  show the empirical kernel
 density estimates (in solid lines) of $\sqrt p \{l_1-42.67\}$ and
 $\sqrt p \{l_p-0.07\}$ from 1000 independent replications, compared
 to their Gaussian limits $N(0, 4246.8)$ and $N(0, 7.2\times10^{-3})$,
 respectively (dashed lines). When considering the empirical
 distribution of the two-dimensional
 random vector
 $  \sqrt p \{l_{p-2}-0.13,l_{p-1}-0.13  \} $,
 we run the two-dimensional kernel density estimation from 1000
 independent replications and display their contour lines, see the
 lower-left panel of the  figure, while the lower-right panel plot
 shows  the
 contour lines of the  kernel density estimation of the eigenvalues of
 the $2 \times 2$ random matrix $-0.138\cdot GOE$ (their limits).

\subsubsection{Binary case}

\begin{figure}[htbp!]
  \centering
  \vspace{1cm}
  \includegraphics[width=0.9\linewidth,height=5cm]{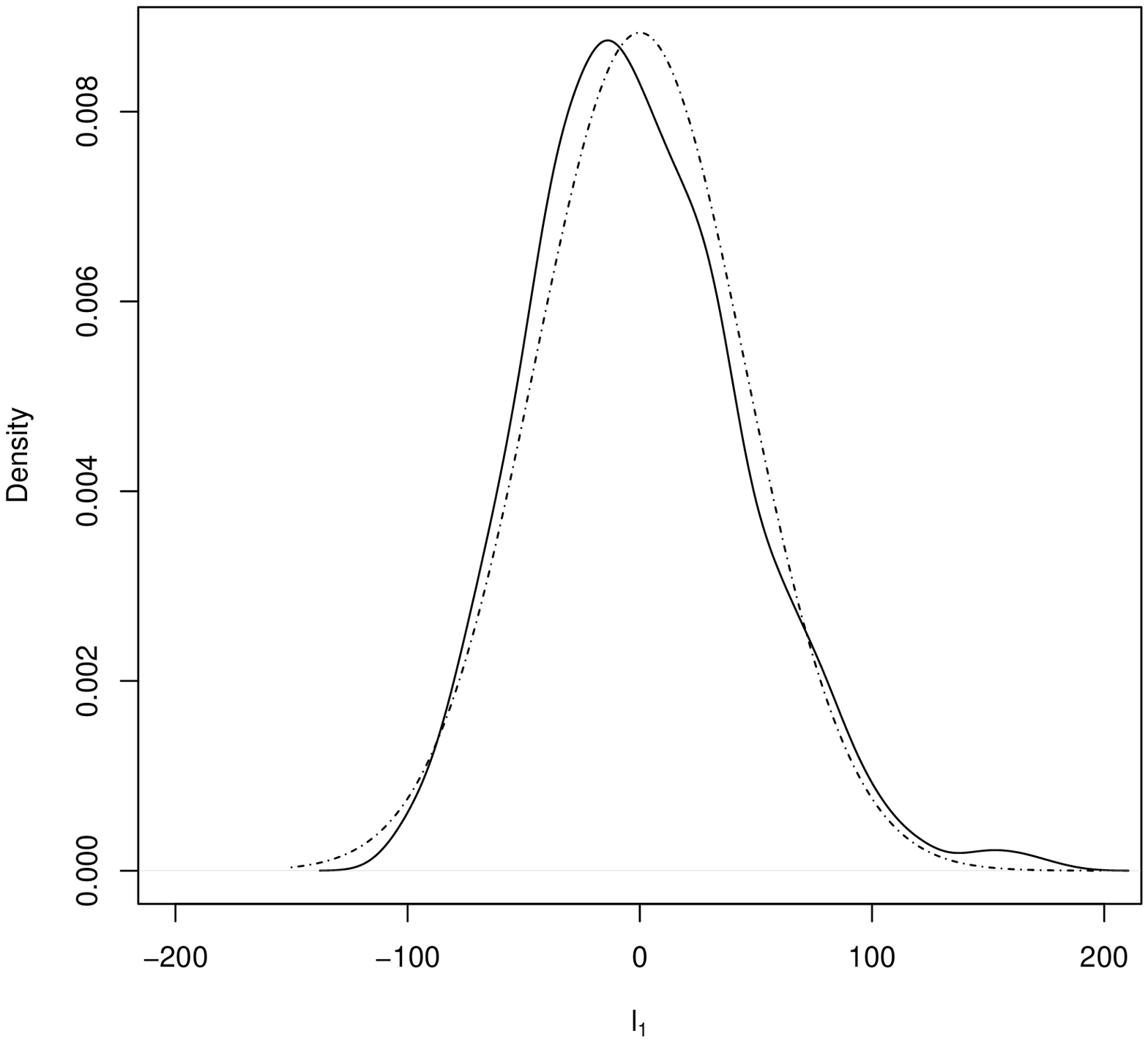}

  \includegraphics[width=0.9\linewidth,height=5cm]{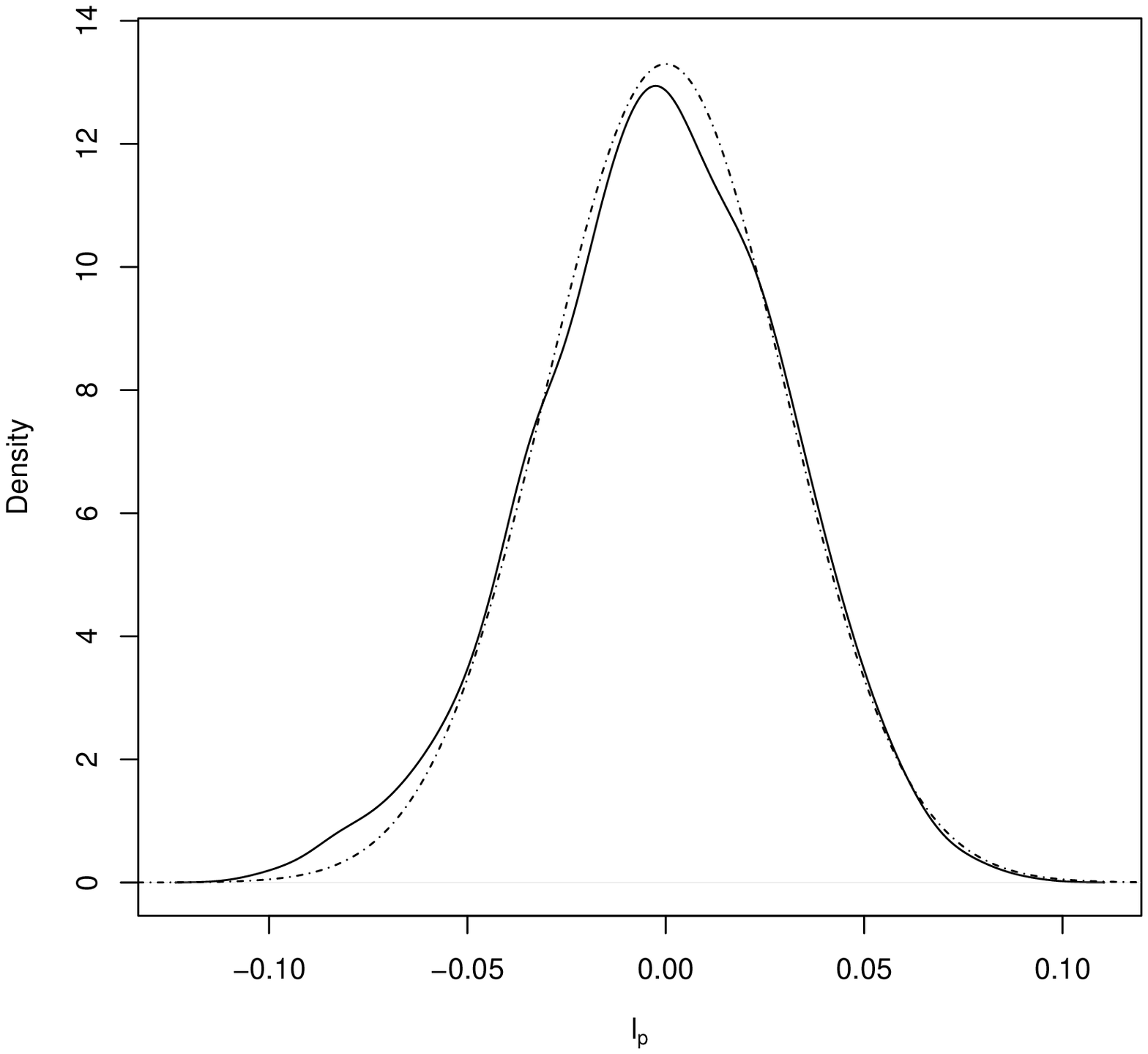}

  \mbox{
  \includegraphics[width=0.5\linewidth]{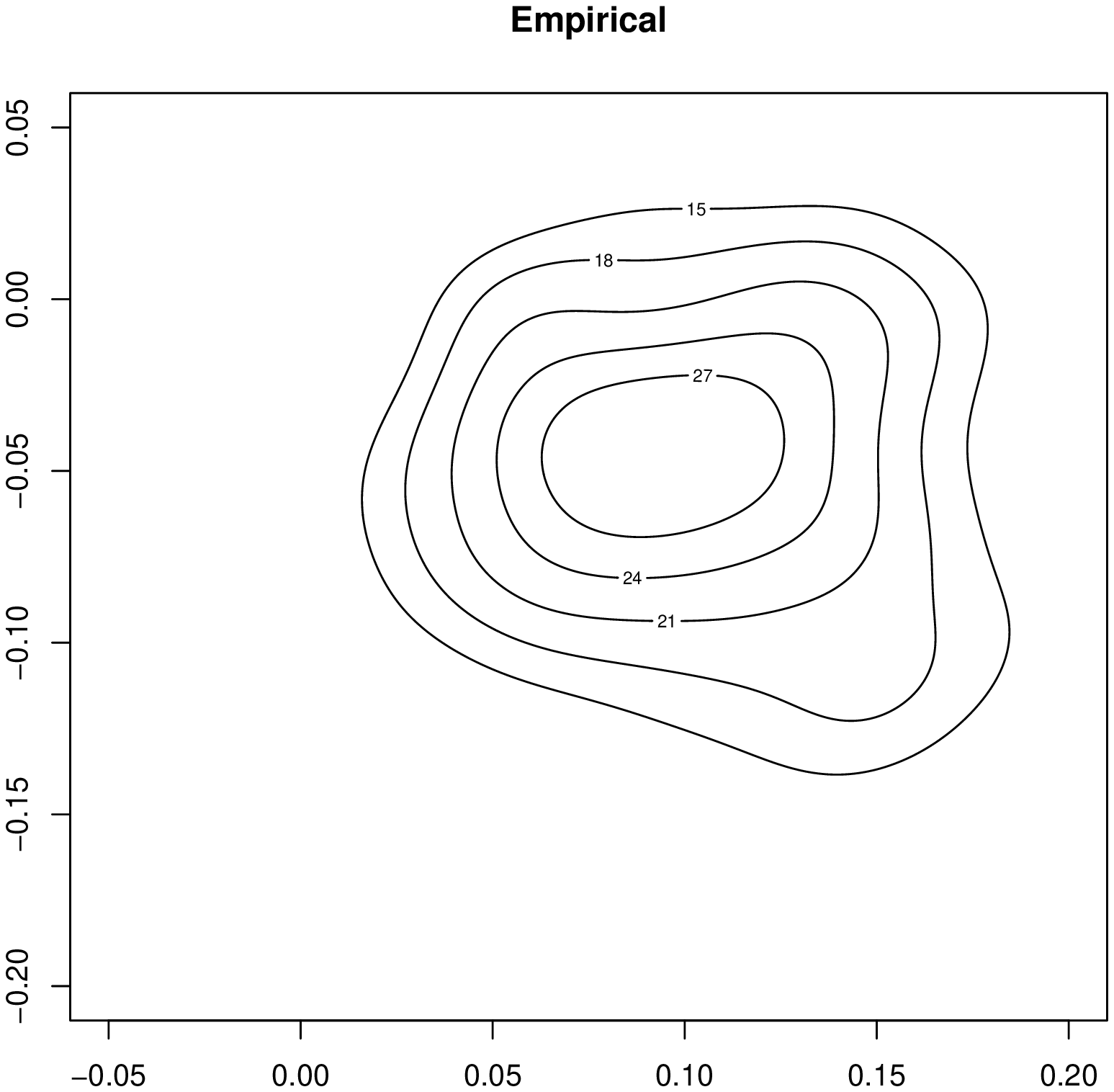}
  \includegraphics[width=0.5\linewidth]{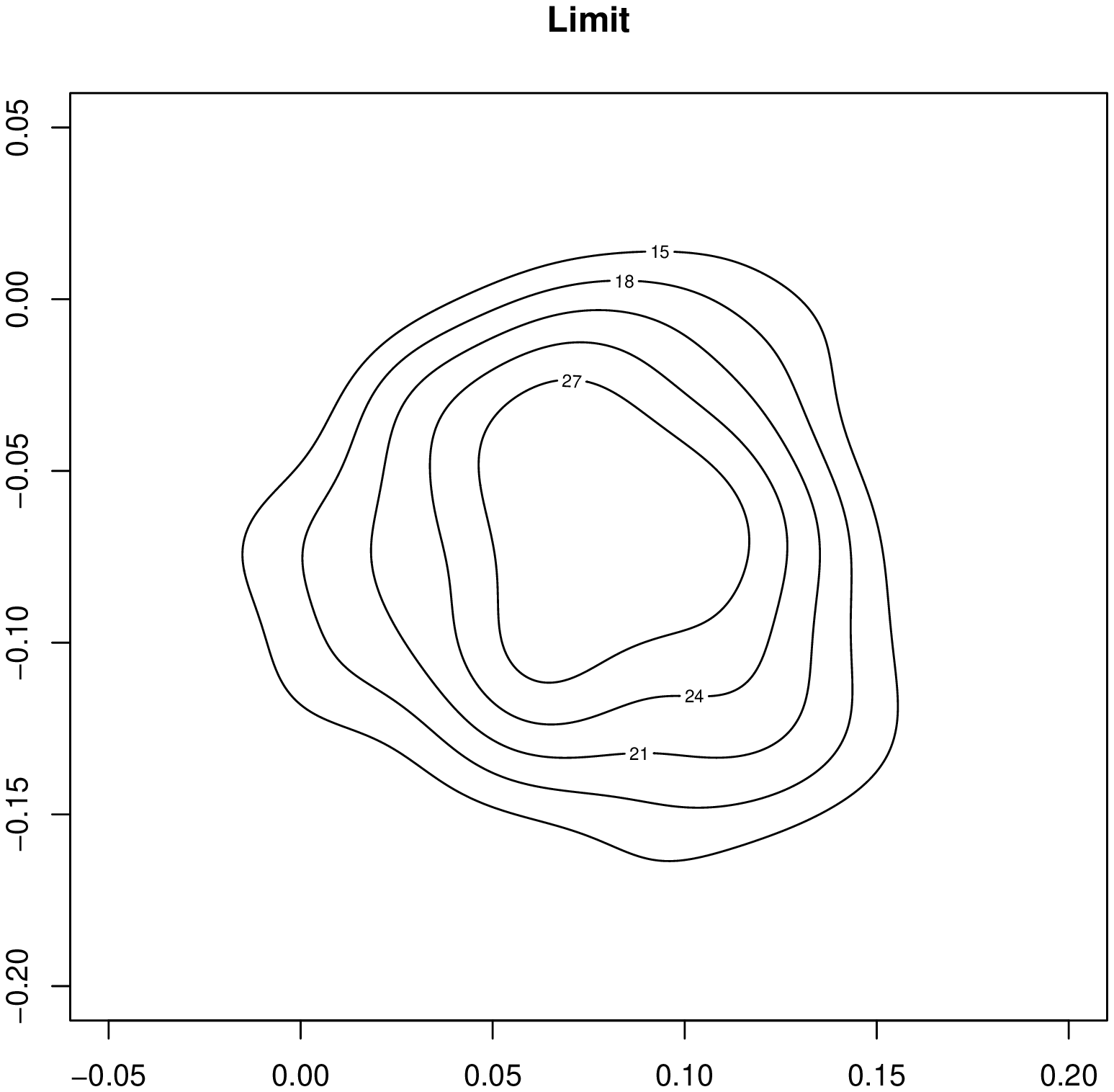}
  }

\caption{Upper panels show the empirical densities  of $l_1$ and
    $l_p$ (solid lines, after centralisation and scaling)
    compared to their Gaussian limits (dashed lines).
    Lower panels show contour plots of  empirical joint  density
    function  of $(l_{p-2}, l_{p-1})$ (left plot, after centralisation and scaling)
    and contour plots of their limits  (right plot). Both the empirical and limit joint density functions are displayed using the two-dimensional kernel density estimates.
    Samples are from i.i.d. binary  distribution with
    $U=I_4$ and
    1000 independent replications.
    \label{binaryextreme}}
\end{figure}
Second, we assume all the  $z_{ij}$ and $w_{ij}$ are i.i.d. binary variables taking values $\{1, -1\}$ with probability $1/2$, and in this case we have $v_4=1$. Similarly, we have
 \begin{itemize}
   \item $\sqrt p \{l_1-42.67\}\rightarrow N(0, 2039.8)$~,
   \item $\sqrt p \{l_p-0.07\}\rightarrow N(0, 9\times10^{-4})$~,
   \item The two-dimensional random vector
     $  \sqrt p \{l_{p-2}-0.13,l_{p-1}-0.13  \} $
     converges to the eigenvalues of the random matrix  $-R_{mn}/1.45$.  Here, $R_{mn}$ is the $2 \times 2$ symmetric random matrix, made with independent Gaussian entries of mean zero and variance
      \begin{align*}
         \var (R_{mn})=\left\{\begin{array}{ll}
        0.008~,& m=n~,\\
        0.02~, & m\neq n~.
      \end{array}\right.
  \end{align*}
 \end{itemize}
Figure \ref{binaryextreme}, upper panels,  show the empirical kernel density
estimates  of $\sqrt p \{l_1-42.67\}$ and $\sqrt p \{l_p-0.07\}$ from
1000 independent replications (in solid lines), compared to their
Gaussian limits (in dashed lines). Also, the lower panel on the figure
show  the
contour lines of the  empirical joint  density of the
$  \sqrt p \{l_{p-2}-0.13,l_{p-1}-0.13  \} $
(the left plot), with the right plot displaying   the contour lines of their  limit.

\subsection{Case of general U}

In this subsection,   we consider  the following non unit 
orthogonal matrix 
\begin{align}\label{u}U=\begin{pmatrix}
    1 & 0 & 0 & 0 \\
    0 & 1 & 0 & 0 \\
    0 & 0 & \frac{1}{\sqrt 2} & \frac{1}{\sqrt 2} \\
    0 & 0 & \frac{1}{\sqrt 2} & \frac{-1}{\sqrt 2} \\
  \end{pmatrix}~,
\end{align} i.e., we have
\begin{align*}
U_1=\begin{pmatrix}
      1 \\
      0 \\
      0 \\
      0 \\
    \end{pmatrix}~,\quad
U_2=\begin{pmatrix}
      0 & 0 \\
      1 & 0 \\
      0 &  \frac{1}{\sqrt 2}\\
      0 &  \frac{1}{\sqrt 2} \\
    \end{pmatrix}~,\quad
U_3=\begin{pmatrix}
      0 \\
      0 \\
       \frac{1}{\sqrt 2} \\
       \frac{-1}{\sqrt 2} \\
    \end{pmatrix}~.
\end{align*}
Since Gaussian distribution is invariant under orthogonal transformation, we only consider the case that
all the  $z_{ij}$ and $w_{ij}$ to be i.i.d. binary variables taking
values $\{1, -1\}$ with probability $1/2$, with all the other settings
fixed as in  Section~\ref{ui}.
Then according to Theorem \ref{mainth2}, we have
 \begin{itemize}
   \item $\sqrt p \{l_1-42.67\}\rightarrow N(0, 2039.8)$~,
   \item $\sqrt p \{l_p-0.07\}\rightarrow N(0, 0.004)$~,
   \item The two-dimensional random vector
     $  \sqrt p \{l_{p-2}-0.13,l_{p-1}-0.13  \} $
     converges to the eigenvalues of the random matrix $-U^{*}_2R(\lambda_2)U_2/1.45$.  Here, $R(\lambda_2)$ is the $4 \times 4$ symmetric random matrix, made with independent Gaussian entries of mean zero and variance
      \begin{align*}
         \var (R_{mn})=\left\{\begin{array}{ll}
        0.008~,& m=n~,\\
        0.02~, & m\neq n~.
      \end{array}\right.
  \end{align*}
 \end{itemize}
 Figure \ref{generalextreme}, upper panels,  show the empirical kernel density
 estimates  of $\sqrt p \{l_1-42.67\}$ and $\sqrt p \{l_p-0.07\}$ from
 1000 independent replications (in solid lines), compared to their
 Gaussian limits (in dashed lines). Also, the lower panel of the figure shows the
 contour lines of the empirical joint   density  of
 $  \sqrt p \{l_{p-2}-0.13,l_{p-1}-0.13  \} $
 (the lower-left plot), with the lower-right plot  showing the contour lines of their  limit.

\begin{figure}[htbp!]
  \centering
  \vspace{1cm}
  \includegraphics[width=0.9\linewidth,height=5cm]{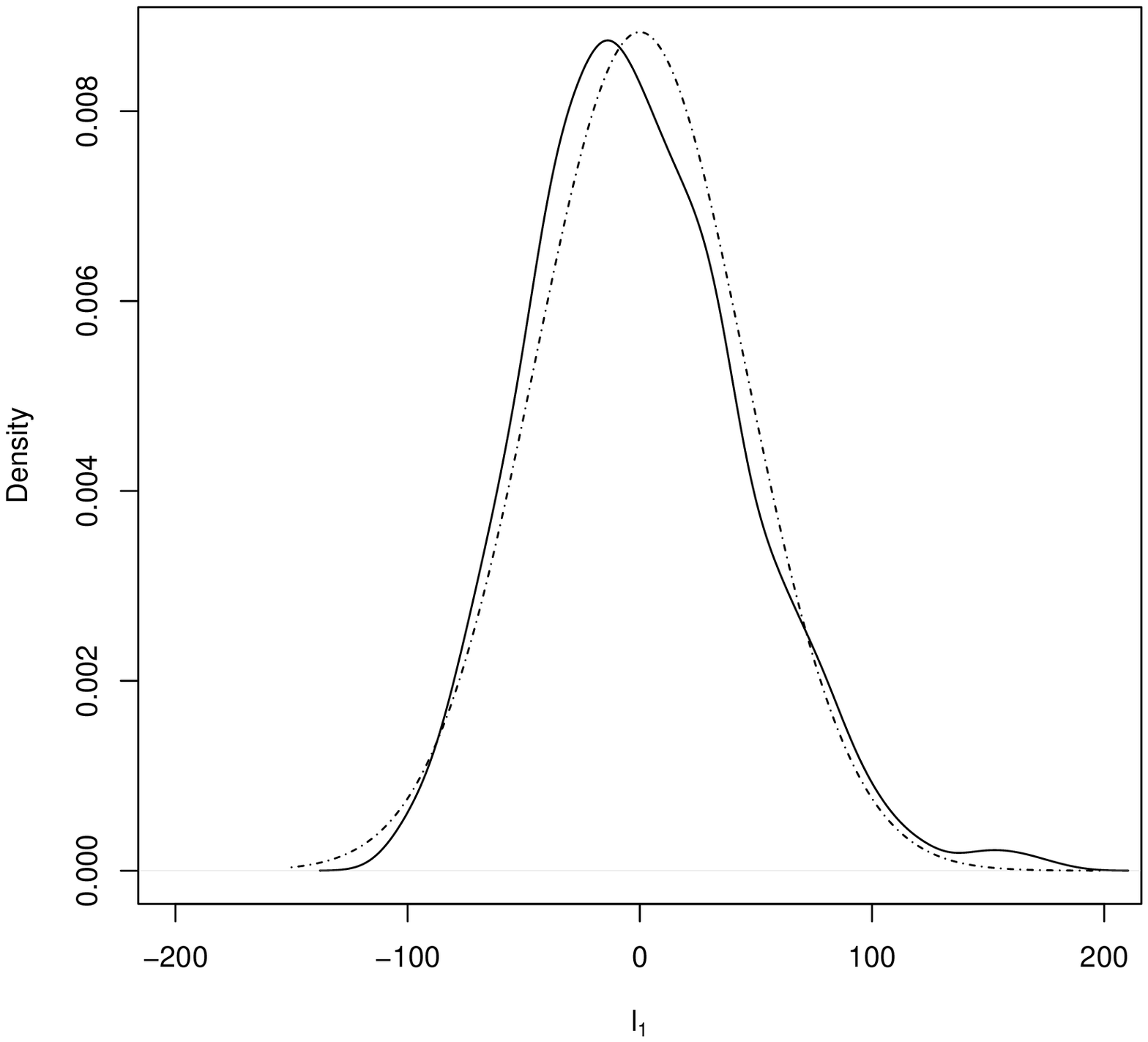}

  \includegraphics[width=0.9\linewidth,height=5cm]{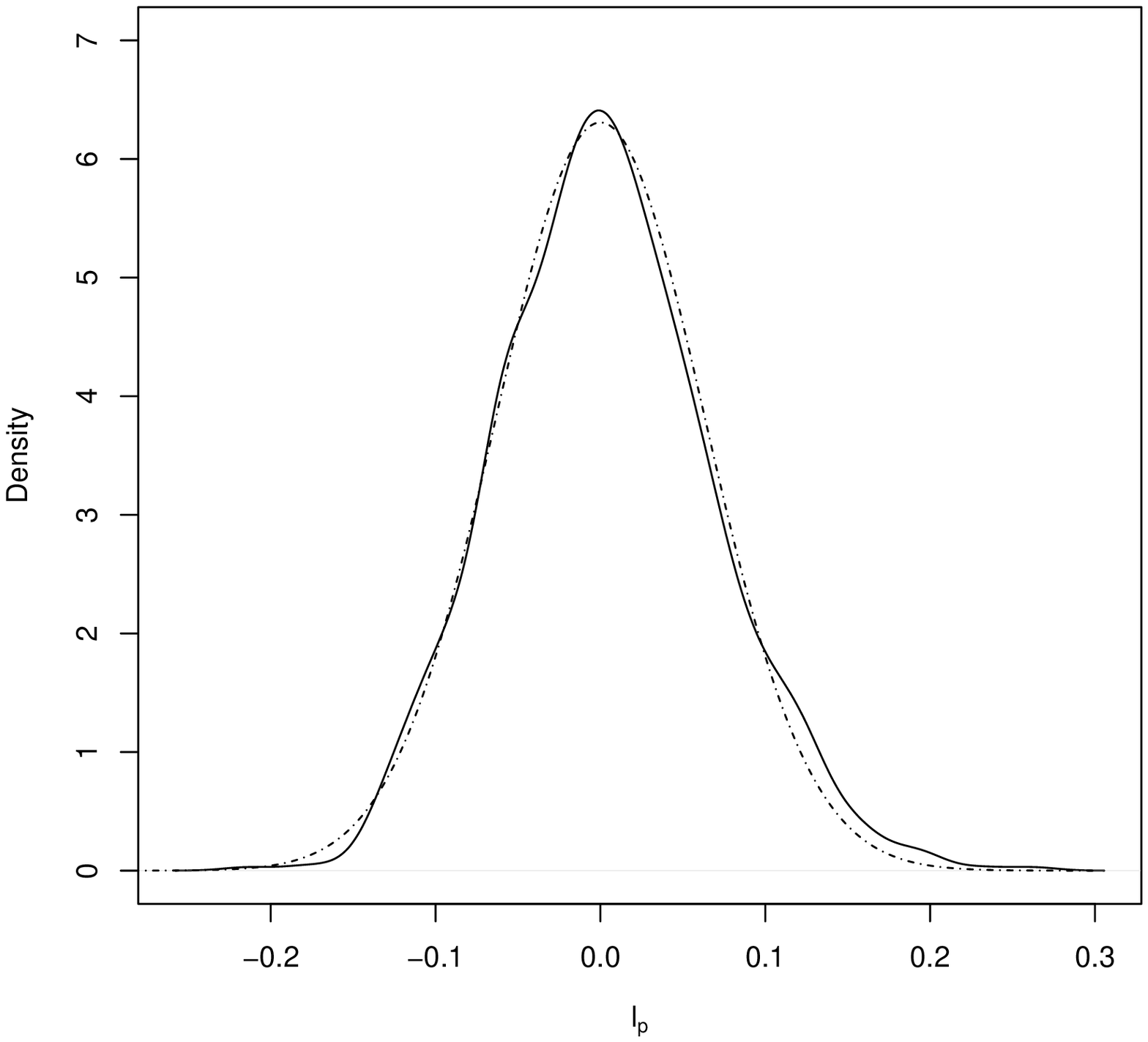}

  \mbox{
    \includegraphics[width=0.5\linewidth]{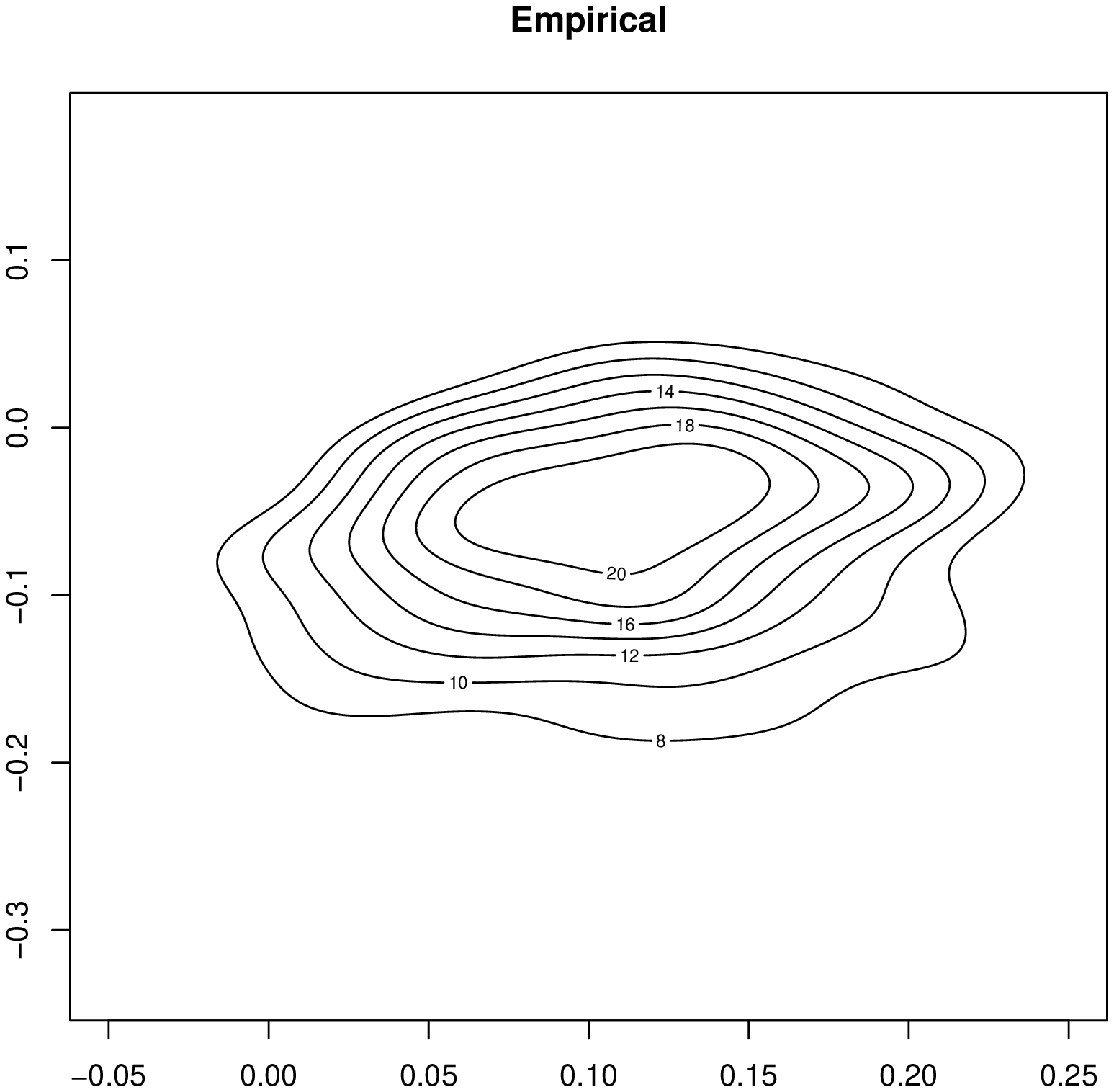}
    \includegraphics[width=0.5\linewidth]{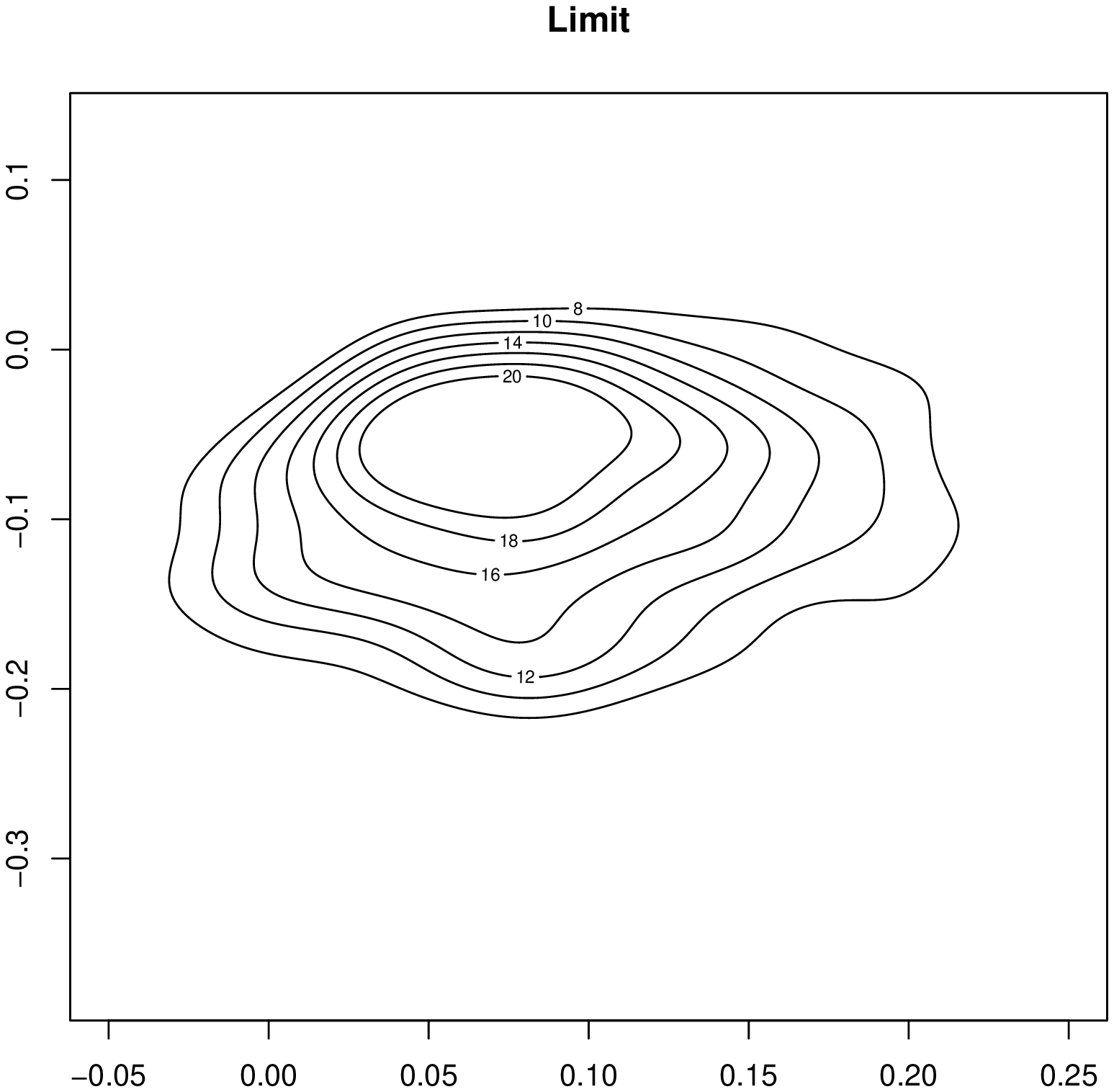}
  }

\caption{Upper panels show the empirical densities  of $l_1$ and
    $l_p$ (solid lines, after centralisation and scaling)
    compared to their Gaussian limits (dashed lines).
    Lower panels show contour plots of  empirical joint  density
    function  of $(l_{p-2}, l_{p-1})$ (left plot, after centralisation and scaling)
    and contour plots of their limits  (right plot). Both the empirical and limit joint density functions are displayed using the two-dimensional kernel density estimates.
    Samples are from i.i.d. binary  distribution with $U$ given by \eqref{u}
    and
    1000 independent replications.
    \label{generalextreme}}
\end{figure}

\section{Joint distribution of the outlier eigenvalues}
\label{sec:CLT2}

In the previous section, we have obtained the following result for the
outlier eigenvalues: the $n_i$-dimensional real random vector $\sqrt p
\{l_{p,j}-\lambda_i, j \in J_i\}$ converges to the distribution of the
eigenvalues of random matrix
$-U^{*}_iR(\lambda_i)U_i/\Delta(\lambda_i)$. It is in fact possible to
derive their joint distribution, i.e. the  limit of the
$M$-dimensional real random vector 
\begin{align}\label{vect}
  \left(\begin{array}{c}
      \sqrt p \{l_{p,j_1}-\lambda_1, j_1 \in J_1\}\\
      \vdots\\
      \sqrt p \{l_{p,j_k}-\lambda_k, j_k \in J_k\}
    \end{array}\right)~.
\end{align}
Such joint convergence results are useful for inference procedures
where consecutive sample eigenvalues are used such as their
differences or ratios, see e.g. \citet{o09} and \citet{PasYao14}.

\begin{theorem}\label{mainth3}
  Assume the same condition as in Theorem~\ref{mainth2} and that all the
  population spikes $a_i$ satisfy
  the condition  $|a_i-\gamma|>\gamma \sqrt{c+y-cy}$. Then the
  $M$-dimensional vector in \eqref{vect}  converges in distribution
  to the eigenvalues of the  $M \times M$ random matrix
  \begin{align}\label{kblock}
    \left(\begin{array}{ccc}
        \frac{-U^{*}_1R(\lambda_1)U_1}{\Delta(\lambda_1)} &\ldots&0\\
        \vdots&\ddots &\vdots\\
        0&\cdots & \frac{-U^{*}_kR(\lambda_k)U_k}{\Delta(\lambda_k)}
      \end{array}\right)~,
  \end{align}
  where the matrices  $R(\lambda_i)$, made with zero-mean independent
  Gaussian random variables,   are defined in
  Theorem~\ref{mainth2}, 
  with the 
  the following covariance function between different blocks ($l\neq s$): for $1\leq i \leq j\leq M$,
  \begin{align*}
    \cov(R(\lambda_l)(i,j), R(\lambda_s)(i,j))=\left\{
      \begin{array}{ll}
        \theta(l,s)~, & i\neq j~,\\
        \omega(l,s)(v_4-3)+2\theta(l,s)~, & i=j~,
      \end{array}\right.
  \end{align*}
  where
  \begin{align*}
    \theta(l,s)&=\lim \frac{1}{n+T} \tr A_n(\lambda_l)A_n(\lambda_s)~,\\
    \omega(l,s)&=\lim \frac{1}{n+T}\sum_{i=1}^{n+T}A_n(\lambda_l)(i,i)A_n(\lambda_s)(i,i)~,
  \end{align*}
  and $A_n(\lambda)$ is defined in \eqref{anl}.
\end{theorem}

The proof of this theorem is very close to that of  Theorem 2.3 in \citet{wang}, thus omitted.

In principle, the limiting parameters $\theta(l,s)$ and
$\omega(l,s)$ can be  completely specified for a given spiked structure. However, this will lead to quite complex formula. Here, we prefer explain a simple case where
 $\Omega_p$ is diagonal whose eigenvalues $|a_i-\gamma|>\gamma \sqrt{c+y-cy}$ ($i=1, \cdots, M$) are all simple, we have $U=I_M$, $M=k$ and $n_i=1$ ($i=1, \cdots, M$).
Therefore, $U^{*}_iR(\lambda_i)U_i$ in \eqref{kblock} reduces to the $(i,i)$-th element of $R(\lambda_i)$, which is a Gaussian random variable. Besides, from Theorem \ref{mainth3}, we see that  the random variables $\{R(\lambda_i)(i,i)\}_{i=1, \cdots, M}$ are jointly independent since the index sets $(i,i)$ are disjoint.
Finally, we have the following joint distribution of the  $M$ outlier eigenvalues of $S_2^{-1}S_1$.
\begin{proposition}\label{jointsim}
  Under the same assumptions as in Theorem \ref{mainth2}, when $\Omega_p$ is diagonal with all its eigenvalues $|a_i-\gamma|>\gamma \sqrt{c+y-cy}$ being simple, the $M$ outlier eigenvalues  $l_{p,j}$ ($j=1, \cdots, M$) of $S_2^{-1}S_1$ are asymptotically independent Gaussian:
  \begin{align*}
    \left(\begin{array}{c}
        \sqrt p (l_{p,1}-\lambda_1)\\
        \vdots\\
        \sqrt p (l_{p,M}-\lambda_M)
      \end{array}\right)\Longrightarrow \mathcal{N}\left({\bf 0}_M, \begin{pmatrix}
        \sigma^2_1 &\cdots & 0 \\
        \vdots & \ddots & \vdots \\
        0 & \cdots & \sigma^2_M \\
      \end{pmatrix}
    \right)~,
  \end{align*}
  where
  \begin{align*}
    \sigma^2_i&=\frac{2a^2_i(cy-c-y)(a_i-1)^2(-1+2a_i+c+a^2_i(y-1))}{(1+a_i(y-1))^4}\\
    &\quad+(v_4-3)\cdot\frac{a^2_i(c+y)(-1+2a_i+c+a^2_i(y-1))^2}{(1+a_i(y-1))^4}~.
  \end{align*}
\end{proposition}

\section{Application to large-dimensional signal detection}
\label{sec:app}

In this section,  we develop  an application of the previous results
to an inference problem where  spiked Fisher matrices
arise naturally. In a signal detection equipment, records are of form
\begin{equation}
  \label{eq:signal}
  x_i = A s_i + e_i ~,\quad i=1,\ldots,T
\end{equation}
where $x_i$ is $p$-dimensional, $s_i$ is a $k\times 1$ low-dimensional signal
$(k\ll p)$ with unit covariance matrix, $A$ a $p\times k$ {\em mixing}
matrix, and $(e_i)$ is an i.i.d. noise with covariance matrix
$\Sigma_2$.
Therefore, the covariance matrix of $x_i$ can be considered as a
$k$-dimensional perturbation of $\Sigma_2$, denoted as $\Sigma_p$ in
the following. Notice that none of the quantities in the r.h.s. of
\eqref{eq:signal}
is observed.
One of the fundamental problem here is to estimate $k$, the number of
signals present in the system.  This problem is challenging when the
dimension $p$ is large, say has a comparable magnitude with the sample
size $T$.  When the noise has the simplest
covariance structure, i.e. $\Sigma_2=\sigma^2_eI_p$, this
problem has been much investigated recently and several solutions are
proposed, see e.g. \cite{Nadler1}, \cite{Nadler3},
\cite{PasYao12,PasYao14}.
However the problem with an {\em arbitrary} noise covariance matrix $\Sigma_2$, say  diagonal to simplify,
remains unsolved in the large-dimensional
context (to the best of our knowledge).  Nevertheless, there exists
 an astute engineering
device where
the system can be tuned in a signal-free
environment, for example in laboratory: that is  we can directly record
 a sequence of pure-noise observations $z_j$, $j=1,\ldots,n$,
which have the same distribution as the
$(e_i)$ above.
These signal-free records can then be used to  whiten the
 observations $(x_i)$  as follows.
Let $S_1 =
T^{-1}  \sum_{i=1}^T x_i x_i^*$, $S_2 =  n^{-1}  \sum_{i=1}^n z_i z_i^*$ and
$l_i, i=1,\cdots, p$ be the eigenvalues of $S_2^{-1}S_1$.
Notice that the  eigenvalues $\{l_i\}$
are invariant  under the transformation
$S_1\mapsto \Sigma_2^{-1/2} S_1\Sigma_2^{-1/2}$,
$S_2\mapsto \Sigma_2^{-1/2} S_2\Sigma_2^{-1/2}$; they are in fact {\em
  independent of  $\Sigma_2$}.
Therefore, these eigenvalues
can be thought as if $\Sigma_2=I_p$, that is
$S_2^{-1}S_1$ becomes a spiked Fisher matrix as introduced in Section
\ref{sec:model}. This  is actually the reason  why the two sample
procedure developed here can deal with an arbitrary covariance matrix
of the  noise while the existing one-sample procedures cannot.
Based on  Theorem~\ref{mainth1}, we propose our estimator of the
number of signals as the number of  eigenvalues of $S_2^{-1}S_1$
 larger than the right edge point of the support of its LSD:
\begin{align}\label{estimator}
\hat{k}=\max\{i: l_i\geq b+d_n\}~,
\end{align}
where $(d_n)$ is a sequence of vanishing constants.

\begin{theorem}\label{appli}
Assume all the spike eigenvalues $a_i$ ($i=1, \cdots, k$) satisfy $a_i>\gamma+\gamma \sqrt{c+y-cy}$. Let $d_n$ be a  sequence of positive numbers such that $d_n \rightarrow 0$, $\sqrt p \cdot d_n \rightarrow 0$ and $p^{2/3}\cdot d_n \rightarrow +\infty$ as $p \rightarrow +\infty$, then the estimator $\hat{k}$ is  constant, i.e. $\hat{k} \rightarrow k$ in probability as $p \rightarrow +\infty$.
\end{theorem}

\begin{remark}
Notice here that there's no need for those spikes $a_i$ to be simple, the only requirement is that they should be properly strong enough $(a_i>\gamma+\gamma \sqrt{c+y-cy})$ for detection.
\end{remark}

\begin{proof}(of Theorem \ref{appli}).\quad  Since
\begin{align*}
\{\hat{k}=k\}=\big\{k=\max\{i: l_i\geq b+d_n\}\big\}=\big\{\forall j \in \{1, \cdots, k\}, l_j\geq b+d_n\big\}\bigcap \big\{l_{k+1}<b+d_n\big\}~,
\end{align*}
we have
\begin{align}\label{r5}
P\{\hat{k}=k\}&=P\left(\bigcap_{1\leq j \leq k}\{ l_j\geq b+d_n\big\}\bigcap \big\{l_{k+1}<b+d_n\big\}\right)\nonumber\\
&=1-P\left(\bigcup_{1\leq j \leq k}\{ l_j< b+d_n\big\}\bigcup \big\{l_{k+1}\geq b+d_n\big\}\right)\nonumber\\
&\geq 1-\sum_{j=1}^k P(l_j< b+d_n)-P(l_{k+1}\geq b+d_n)~.
\end{align}
For $j=1, \cdots, k$,
\begin{align}\label{r1new}
P(l_j< b+d_n)&=P\Big(\sqrt p (l_j-\phi(a_j))< \sqrt p(b+d_n-\phi(a_j))\Big)\nonumber\\
&\rightarrow P\Big(\sqrt p (l_j-\phi(a_j))< \sqrt p(b-\phi(a_j))\Big)~,
\end{align}
which is due to the assumption that $\sqrt p \cdot d_n \rightarrow
0$. Then the part $\sqrt p(b-\phi(a_j))$ in \eqref{r1new} will tend to
$-\infty$ since we have always $\phi (a_j)>b$ when $a_i>\gamma+\gamma
\sqrt{c+y-cy}$. On the other hand, by Theorem~\ref{mainth2},
$\sqrt p (l_j-\phi(a_j))$ in \eqref{r1new} has a limiting
distribution; it is then bounded in probability.  Therefore, we have
\begin{align}\label{r4}
P(l_j< b+d_n)\rightarrow 0~\quad \quad \text{for} ~j=1, \cdots, k~.
\end{align}
Also
\begin{align*}
P(l_{k+1}\geq b+d_n)=P\Big(p^{2/3} (l_{k+1}-b)\geq p^{2/3}\cdot d_n\Big)~,
\end{align*}
and the part $p^{2/3} (l_{k+1}-b)$ is asymptotically Tracy-Widom
distributed (see \citet{bao} where the Tracy-Widom distribution for the largest eigenvalue of general sample covariance matrix is derived). As  $p^{2/3} \cdot d_n$  tend to infinity as assumed,  we have
\begin{align}\label{r3}
P(l_{k+1}\geq b+d_n)=0~.
\end{align}
Combine \eqref{r5}, \eqref{r4} and \eqref{r3}, we have $P\{\hat{k}=k\}\rightarrow 1$ as $p \rightarrow +\infty$.
The proof of Theorem \ref{appli} is complete.
\end{proof}

We conduct a short simulation  to illustrate the performance of our
estimator. We fix $y=1/2$ and $c=1/5$ as in Section~\ref{ni}, and the value of $p$ varies from $50$ to $250$, therefore, the critical value for $a_i$ in the model \eqref{eq:Sigma} (after whitening) is $a_i>\gamma \{1+\sqrt{c+y-cy}\}=3.55$. For each given pair of $(p,n,T)$, we repeat $1000$ times. The tuning parameter $d_n$ is chosen to be $(\log \log p)/p^{2/3}$.

 Next, suppose $k=3$ and $A$ is a $p \times3$ matrix of form $A=(\sqrt{c_1}v_1, \sqrt{c_2}v_2)$, where
$c_1=10$, $c_2=5$,
\begin{align*}
v_1=\begin{pmatrix}
1 & 0 & \cdots & 0 \\
\end{pmatrix}^*
\quad \text{and}\quad v_2=\begin{pmatrix}
             0 & 1/\sqrt{2} & 1/\sqrt{2} & 0 & \cdots & 0 \\
             0 & 1/\sqrt{2} & -1/\sqrt{2} & 0 & \cdots & 0 \\
           \end{pmatrix}^*~.
\end{align*}
 Besides, assume $\cov (s_i)=I_k$.   In this setting, we have two spike eigenvalues $c_1=10$, $c_2=5$ (before whitening) with multiplicity $n_1=1$, $n_2=2$, respectively.
Finally, we choose $\cov(e_i)$ to be either diagonal or non-diagonal as below.
\begin{table}[htbp]
\centering
\caption{\label{model11} Frequency of our estimator  in Model 1.}
\begin{tabular}{cccccccccccc}
\hline
$p$ && &50 && 100 && 150 && 200 && 250\\
$n$ && &100&& 200&& 300 && 400 && 500\\
$T$ & &&250 && 500 && 750 && 1000 && 1250 \\
\hline
$\hat{k}=1$ && & 0.038 && 0.003 && 0 && 0 && 0.001 \\
$\hat{k}=2$  &&& 0.578 && 0.317 && 0.166 && 0.103 && 0.047\\
${\hat{k}=3}$  &&& \bf{0.381} && \bf{0.675} && \bf{0.818} && \bf{0.883} && \bf{0.937}\\
$\hat{k}=4$  &&& 0.003 && 0.005 && 0.016 && 0.014 && 0.015\\
\hline
\end{tabular}
\end{table}
\begin{itemize}
\item For Model 1: set $\cov(e_i)=\diag(\underbrace{1, \cdots, 1}_{p/2}, \underbrace{2, \cdots, 2}_{p/2})$~. In this case, we have the three non-zero eigenvalues of $(c_1 v_1v^*_1+c_2v_2v^*_2)\cdot[\cov(e_i)]^{-1}$ equal $10, 5, 5$, respectively, which are all larger than the critical value $3.55-1$, therefore, the  number of detectable signals is three;
  \\
\item For Model 2: set $\cov(e_i)$
  be  compound symmetric with all the diagonal elements equal $1$ and all the off-diagonal elements equal $0.1$. In this case, we have for each given $p$, the  three non-zero eigenvalues of $(c_1 v_1v^*_1+c_2v_2v^*_2)\cdot[\cov(e_i)]^{-1}$  are all larger than $5.36(>3.55-1)$. The number of detectable signals is again three.
\end{itemize}
\begin{table}[htbp]
\centering
\caption{\label{model2} Frequency of our estimator  in Model 2.}
\begin{tabular}{cccccccccccc}
\hline
$p$ && &50 && 100 && 150 && 200 && 250\\
$n$ && &100&& 200&& 300 && 400 && 500\\
$T$ & &&250 && 500 && 750 && 1000 && 1250 \\
\hline
$\hat{k}=1$ && & 0.016 && 0 && 0 && 0 && 0 \\
$\hat{k}=2$  &&& 0.475 && 0.186 && 0.053 && 0.028 && 0.008\\
${ \hat{k}}=3$  &&& \bf{0.505} && \bf{0.806} && \bf{0.926} && \bf{0.950} && {\bf 0.971}\\
$\hat{k}=4$  &&& 0.004 && 0.008 && 0.021 && 0.022 && 0.021\\
\hline
\end{tabular}
\end{table}

 \noindent Tables \ref{model11} and  \ref{model2} report the empirical
 frequency of our estimator $\hat{k}=1, 2, 3 , 4$ in Model 1 and Model 2, respectively, where the true number of signals is $k=3$. Also, Figure \ref{codiag} shows more clearly the trends of the frequency of correct estimation in both cases. We can see the frequency both  increase as $p$ gets larger, which confirms the consistency of our estimator.
\begin{figure}[htfp]
  \centering
  \vspace{1cm}
  \mbox{
  \includegraphics[width=0.5\linewidth]{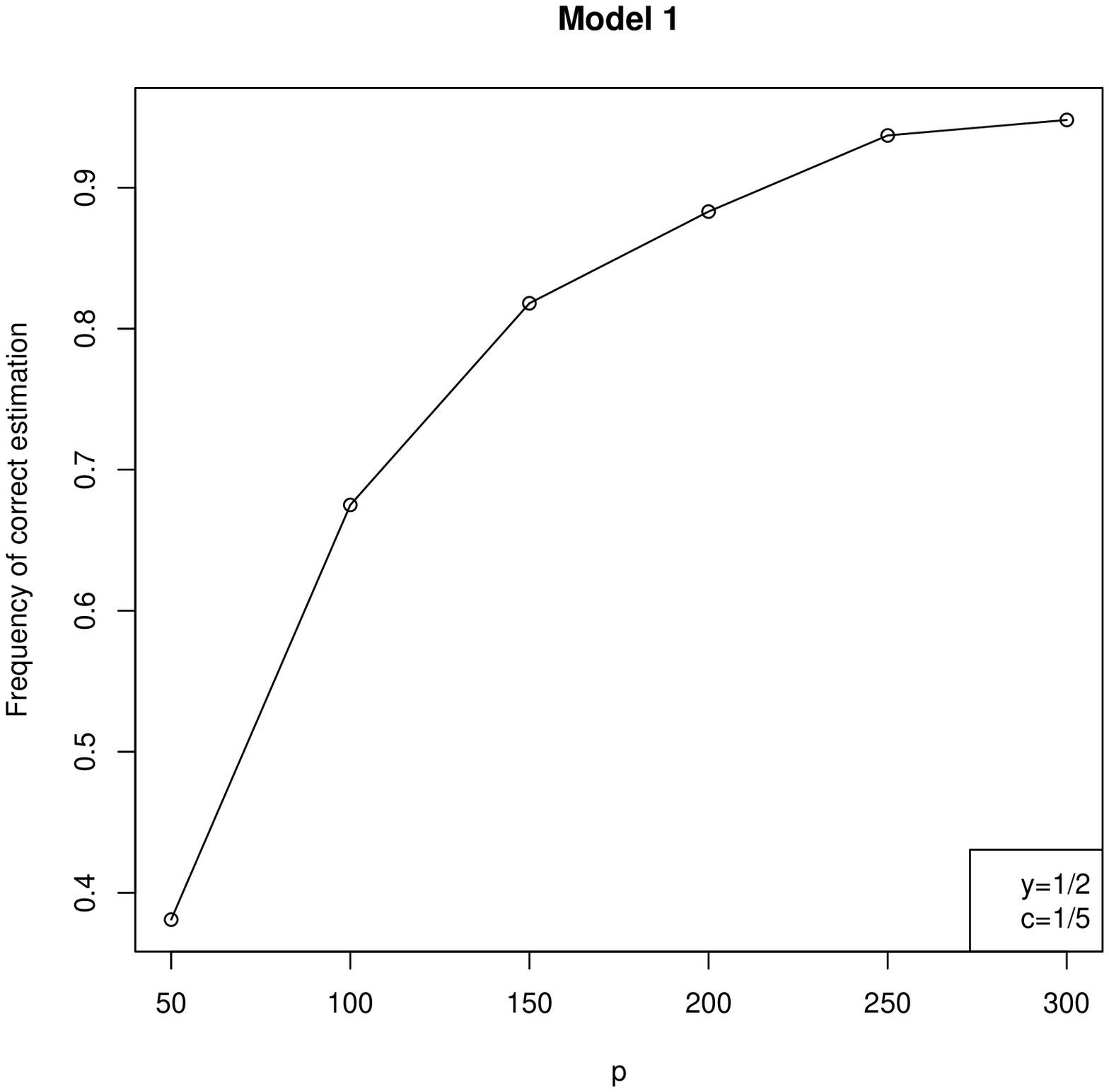}
  \includegraphics[width=0.5\linewidth]{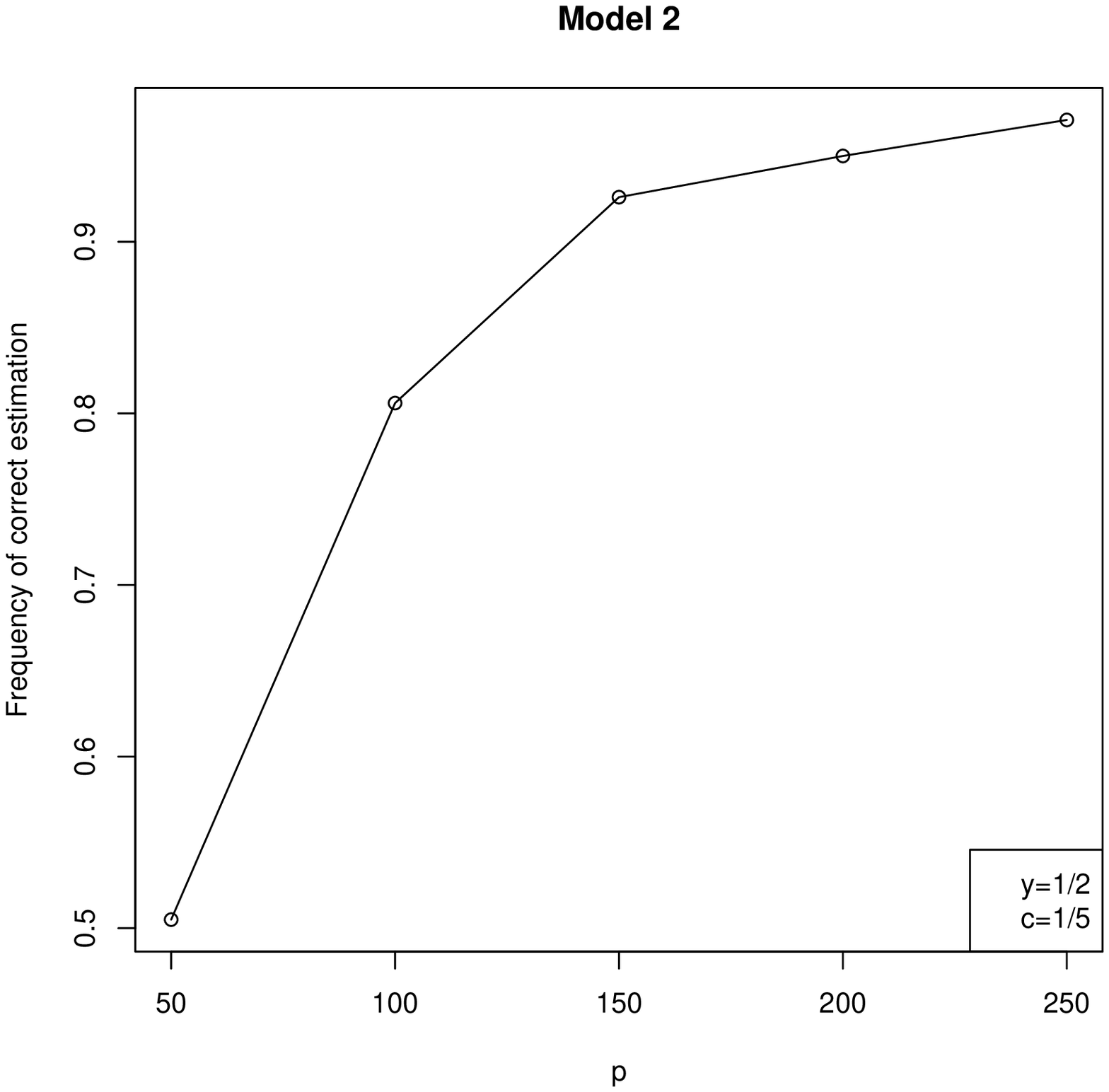}
  }
  \caption{Frequency of true estimation $\hat k=k=3$.
  \label{codiag}}
\end{figure}

\appendix
\section{Some lemmas}
\label{sec:lemmas}

\begin{lemma}\label{unique}
Let R be a $M\times M$ real-valued matrix,  $U=\begin{pmatrix}
                                                    U_1 & \cdots & U_k \\
                                                  \end{pmatrix}
$ and $V=\begin{pmatrix}
                                                    V_1 & \cdots & V_k \\
                                                  \end{pmatrix}
$ are two orthogonal bases of some subspace $E\subseteq \R^M$ of dimension $M$, where both $U_i$ and $V_i$ are of size $M \times n_i$, satisfying $n_1+\cdots n_k=M$. Then the two $n_i \times n_i$ matrices $U^*_iRU_i$ and $V^*_iRV_i$ have the same eigenvalues.
\end{lemma}

\begin{proof}(of Lemma \ref{unique})
It is sufficient to prove that there exists a $n_i \times n_i$ orthogonal matrix $A$, such that $V_i=U_i\cdot A$. If it is true, then
$V^*_iRV_i=A^*(U^*_iRU_i)A$, and since $A$ is orthogonal, we have the eigenvalues of $V^*_iRV_i$ and $U^*_iRU_i$ are the same.
Now let $U_i=\begin{pmatrix}
                  u_1 & \cdots & u_{n_i}\\
                \end{pmatrix}
$ and
$V_i=\begin{pmatrix}
                  v_1 & \cdots & v_{n_i}\\
                \end{pmatrix}.
$
Define $A=(a_{ls})_{1\leq l,s \leq n_i}$, such that
\begin{align*}
\left\{\begin{array}{c}
v_1=a_{11}u_1+\cdots a_{n_i\,1}u_{n_i}\\
\vdots\\
v_{n_i}=a_{1n_i}u_1+\cdots a_{n_i\,n_i}u_{n_i}
\end{array}\right.~.
\end{align*}
Put in matrix form:
\begin{align*}
\begin{pmatrix}
  v_1 & \cdots & v_{n_i} \\
\end{pmatrix}=\begin{pmatrix}
  u_1 & \cdots & u_{n_i} \\
\end{pmatrix}\begin{pmatrix}
               a_{11} & \cdots & a_{1\,n_i} \\
               \vdots & \ddots & \vdots \\
               0 & \cdots & a_{n_i\,n_i} \\
             \end{pmatrix}~,
\end{align*}
i.e. $V_i=U_i\cdot A$.
Since $<v_i, v_j>~=~<a_{\cdot i}, a_{\cdot j}>$ by orthogonality of $\{u_j\}$, where $a_{\cdot k}=(a_{lk})_{1\leq k \leq n_i}$, therefore, the matrix $A$ is orthogonal.
\end{proof}

\begin{lemma}\label{f1}
  Suppose $X=(x_1, \cdots, x_n)$ is a  $p \times n$ matrix, whose columns $\{x_i\}$ are independent random vectors. $Y=(y_1, \cdots, y_n)$ is also similarly defined. Let $\Sigma_p$ be the covariance matrix of $x_i$ and $y_i$,  $A$ is a deterministic matrix, then we have
  \begin{align*}
    XAY^{*}\longrightarrow \tr A \cdot \Sigma_p~.
  \end{align*}
  Moreover, if A is random but independent of $X$ and $Y$, then we have
  \begin{align}\label{ra}
    XAY^{*}\longrightarrow \e\tr A \cdot \Sigma_p~.
  \end{align}
\end{lemma}

\begin{proof}
  We consider the $(i,j)$-th entry of $XAY^{*}$:
  \begin{align}\label{xax}
    XAY^{*}(i,j)=\sum_{k,l=1}^n X(i,k)A(k,l)Y^{*}(l,j)=\sum_{k,l=1}^n X_{ik}Y_{jl}A_{kl}~.
  \end{align}
  Since $X_{ik}Y_{jl}\rightarrow \Sigma_p(i,j)$ when $k=l$. Therefore, the right hand side of \eqref{xax} tends to $\Sigma_p(i,j)\cdot\sum_{k=1}^nA_{kk}$, which is equivalent to say that
  \begin{align*}
    XAY^*\rightarrow \tr A\cdot\Sigma_p~.
  \end{align*}
  Then \eqref{ra} is simply due to the conditional expectation.
  The proof of Lemma \ref{f1} is complete.
\end{proof}

\smallskip
\noindent In all the following, $\lambda$ refers to the outlier limit that
\[
\lambda=\frac{a(a-1+c)}{a-1-ay}~.
\]

\begin{lemma}\label{five}
  We have
  \begin{align*}
    &s(\lambda)=\frac{a(y-1)+1}{(a-1)(a+c-1)}~,\\
    &m_1(\lambda)=\frac{(a(y-1)+1)^2(-1+2a+a^2(y-1)+y(c-1))}{(a-1)^2(a+c-1)^2(-1+2a+c+a^2(y-1))}~,\\
    &m_2(\lambda)=\frac{1}{a-1}~,\\
    &m_3(\lambda)=\frac{-(a(y-1)+1)^2}{(a-1)^2(-1+2a+c+a^2(y-1))}~,\\
    &m_4(\lambda)=\frac{-1+2a+c+a^2(-1+c(y-1))}{(a-1)^2(-1+2a+c+a^2(y-1))}~.
  \end{align*}
\end{lemma}
\begin{proof}(sketch of the proof of Lemma \ref{five})
  In this short proof, we skip all the detailed calculations.
  Recall the definition of $\underline{s}(z)$ in \eqref{stss}, its value at $\lambda$ is
  \begin{align}\label{ds}
    \underline{s}(\lambda)=\frac{a(y-1)+1}{(a-1)(a+c-1)}~.
  \end{align}
  Also,  \eqref{stss} says that $\underline{s}(z)$ is the solution of the following equation:
  \begin{align}\label{us}
    z(c+zy)\underline{s}^2(z)+(c(z(1-y)+1-c)+2zy)\underline{s}(z)+c+y-cy=0~.
  \end{align}
  Taking derivatives on both sides of \eqref{us} and combing with \eqref{ds} will give the value of $\underline{s}'(\lambda)$. On the other hand, since it holds
  \begin{align}
    \underline{s}(z)+\frac 1z (1-c)=cs(z)~,
  \end{align}
  see \eqref{est3}, taking derivatives on both sides again will give the value of $s'(\lambda)$. Finally, the above five values is just a combination of $s(\lambda)$ and $s'(\lambda)$.

  \noindent The proof of Lemma \ref{five} is complete.
\end{proof}

\begin{lemma}\label{l1}
  Under assumptions (a)-(d),
  \[
  \frac 1p\tr \left\{\Big(\lambda\cdot \frac 1n Z_{2}Z^*_{2}-\frac 1T X_2X^*_2\Big)^{-1} \right\}\overset{a.s.}{\longrightarrow} \frac{1}{a+c-1}~.
  \]
\end{lemma}

\begin{proof}(of Lemma \ref{l1})
  We first fix $\frac 1n Z_{2}Z^*_{2}$, then
  we can use the result in \cite{zheng13} (Lemma 4.3), which says that
  \begin{align*}
    \frac 1p \tr\left(\frac 1z \cdot\frac 1n Z_{2}Z^*_{2}-\frac 1T X_2X^*_2\right)^{-1}\rightarrow \tilde{m}(z),\quad a.s.
  \end{align*}
  where $\tilde{m}(z)$ is the unique solution to the equation
  \begin{align}\label{ea1}
    \tilde{m}(z)=\int \frac{1}{\frac{x}{z}-\frac{1}{1-c\tilde{m}(z)}}dF_y(x)
  \end{align}
  satisfying
  \[
  \Im(z)\cdot\Im(\tilde{m}(z))\geq 0~,
  \]
  here, $F_y(x)$ is the LSD of  $\frac 1n Z_{2}Z^*_{2}$ (deterministic), which is the standard M-P law with parameter $y$. Besides, if we denote its Stieltjes transform as $s(z):=\int \frac{1}{x-z}dF_y(x)$, then \eqref{ea1} could be written as
  \begin{align}\label{ms}
    \tilde{m}(z)=\int \frac{z}{x-\frac{z}{1-c\tilde{m}(z)}}dF_y(x)=z\cdot s\left(\frac{z}{1-c\tilde{m}(z)}\right)~.
  \end{align}
  Since we know that the Stieltjes transform of the LSD of a standard sample covariance matrix satisfies:
  \begin{align}\label{ss}
    s(z)=\frac{1}{1-y-yzs(z)-z}~,
  \end{align}
  then we bring \eqref{ms} into \eqref{ss} leads to
  \begin{align*}
    \frac{\tilde{m}(z)}{z}=\frac{1}{1-y-y\cdot\frac{z}{1-c\tilde{m}(z)}\cdot\frac{\tilde{m}(z)}{z}-\frac{z}{1-c\tilde{m}(z)}}~,
  \end{align*}
  whose nonnegative solution is unique, which is
  \begin{align}\label{em1}
    \tilde{m}(z)=\frac{-1+y+z-zc+\sqrt{(1-y-z+zc)^2+4z(yc-y-c)}}{2(yc-y-c)}~.
  \end{align}
  Therefore, we have for fixed $\frac 1n Z_{2}Z^*_{2}$,
  \begin{align*}
    \frac 1p \tr \Big(\lambda\cdot \frac 1n Z_{2}Z^*_{2}-\frac 1T X_2X^*_2\Big)^{-1}\rightarrow \tilde{m}\left(\frac{1}{\lambda}\right)=\frac{1}{a+c-1}~
  \end{align*}
  almost surely.
  Finally, due to the fact that for each $\omega$, the ESD of $\frac 1n Z_{2}Z^*_{2}(\omega)$ will tend to the same limit (standard M-P distribution), which is independent of the choice of $\omega$. Therefore, we have for all $\frac 1n Z_{2}Z^*_{2}$ (not necessarily deterministic but independent of $\frac 1T X_2X^*_2$),
  \begin{align*}
    \frac 1p \tr \Big(\lambda\cdot \frac 1n Z_{2}Z^*_{2}-\frac 1T X_2X^*_2\Big)^{-1}\rightarrow \frac{1}{a+c-1}~
  \end{align*}
  almost surely.

  \noindent The proof of Lemma \ref{l1} is complete.
\end{proof}

\begin{lemma}\label{asl}
  $A(\lambda)$, $B(\lambda)$, $C(\lambda)$ and $D(\lambda)$ are defined in \eqref{abcd}, 
  then
  \begin{align}
    &(l-\lambda)\cdot\frac 1n Z_1A(l)Z^*_1\rightarrow (l-\lambda)\cdot(1+y\lambda s(\lambda))\cdot I_M~,\label{x1}\\
    & \frac {\lambda}{n} Z_1[A(l)-A(\lambda)]Z^*_1\rightarrow (l-\lambda)\cdot(\lambda y s(\lambda)+\lambda^2 y m_1(\lambda))\cdot I_M~,\label{x2}\\
    &\frac 1TX_1(B(l)-B(\lambda))X^*_1\rightarrow -(l-\lambda)\cdot cm_3(\lambda)\cdot U \left(\begin{array}{c}
        a_1 \quad\quad\\
        \quad\ddots \quad\\
        \quad \quad a_k
      \end{array}\right)U^{*}~,\label{x3}\\
    &\frac ln Z_1(C(l)-C(\lambda))X^*_1+\frac {l-\lambda}{n} Z_1C(\lambda)X^*_1\rightarrow (l-\lambda)\cdot {\bf 0}_{M \times M}~,\label{x4}\\
    &\frac lT X_1(D(l)-D(\lambda))Z^*_1+\frac {l-\lambda}{T} X_1D(\lambda)Z^*_1\rightarrow (l-\lambda)\cdot {\bf 0}_{M \times M}~.\label{x5}
  \end{align}
\end{lemma}

\begin{proof}(of Lemma \ref{asl})

  \noindent {\bf Proof of \eqref{x1}:}
  Since $Z_1$ is independent of $A$ and $\cov(Z_1)=I_M$, we combine this fact with Lemma \ref{f1}:
  \begin{align}\label{al}
    (l-\lambda)\cdot\frac 1n Z_1A(l)Z^*_1\rightarrow(l-\lambda)\cdot \frac 1n\e\tr A(l)\cdot I_M~.
  \end{align}
  Considering the expression of $A(l)$, we have
  \begin{align*}
    \frac 1n\e \tr A(\lambda)&=\frac 1n \e \tr \left[I_n-Z^*_2(\lambda I_p-S)^{-1}\Big(\frac 1n Z_2Z^*_2\Big)^{-1}\frac{\lambda}{n}Z_2\right] \nonumber\\
    &=1-\frac{\lambda}{n}\e \tr (\lambda I_p-S)^{-1}\nonumber\\
    &=1-y\lambda\int \frac{1}{\lambda-x}dF_{c,y}(x)\nonumber\\
    &=1+y\lambda s(\lambda)~.
  \end{align*}
  Therefore, combine with \eqref{al}, we have
  \begin{align*}
    (l-\lambda)\cdot\frac 1n Z_1A(l)Z^*_1\rightarrow(l-\lambda)(1+y\lambda s(\lambda))\cdot I_M~.
  \end{align*}
  \noindent {\bf Proof of \eqref{x2}:}
  Bringing the expression of $A(l)$ into consideration, we first have
  \begin{align*}
    &\quad A(l)-A(\lambda)\\
    &=Z^*_2(\lambda I_p-S)^{-1}\Big(\frac 1n Z_2Z^*_2\Big)^{-1}\frac{\lambda}{n}Z_2
    -Z^*_2(l I_p-S)^{-1}\Big(\frac 1n Z_2Z^*_2\Big)^{-1}\frac{l}{n}Z_2\\
    &=Z^*_2(\lambda I_p-S)^{-1}\Big(\frac 1n Z_2Z^*_2\Big)^{-1}\frac{\lambda-l}{n}Z_2
    +Z^*_2\left[(\lambda I_p-S)^{-1}-(l I_p-S)^{-1}\right]\Big(\frac 1n Z_2Z^*_2\Big)^{-1}\frac{l}{n}Z_2\\
    &=(l-\lambda)\cdot\left[-Z^*_2(\lambda I_p-S)^{-1}\Big(\frac 1n Z_2Z^*_2\Big)^{-1}\frac{1}{n}Z_2
      + Z^*_2(\lambda I_p-S)^{-1}(l I_p-S)^{-1}\Big(\frac 1n Z_2Z^*_2\Big)^{-1}\frac{l}{n}Z_2\right]~.
  \end{align*}
  Then using Lemma \ref{f1} for the same reason, we have
  \begin{align*}
    \quad\frac {\lambda}{n} Z_1[A(l)-A(\lambda)]Z^*_1 \rightarrow \frac{\lambda}{n}\left\{\e \tr \left(A(l)-A(\lambda)\right)\right\}\cdot I_M~,
  \end{align*}
  and
  \begin{align*}
    &\quad\frac{1}{n}\e \tr \left(A(l)-A(\lambda)\right)=(l-\lambda)\cdot \bigg[-\frac {1}{n}\e \tr \left\{Z^*_2(\lambda I_p-S)^{-1}\Big(\frac 1n Z_2Z^*_2\Big)^{-1}\frac{1}{n}Z_2\right\}\\
    &\quad\quad\quad\quad \quad\quad\quad\quad+\frac {1}{n} \e \tr\left\{ Z^*_2(\lambda I_p-S)^{-1}(l I_p-S)^{-1}\Big(\frac 1n Z_2Z^*_2\Big)^{-1}\frac{l}{n}Z_2\right\}\bigg]\\
    &=(l-\lambda)\cdot \bigg[-\frac {1}{n}\e \tr (\lambda I_p-S)^{-1}+\frac {\lambda}{n} \e \tr(\lambda I_p-S)^{-2}+o(1)\bigg]\\
    &=(l-\lambda)\cdot\Big[y \int \frac{1}{x-\lambda}dF_{c,y}(x)+\lambda y \int \frac{1}{(\lambda-x)^2}dF_{c,y}(x)+o(1)\Big]\\
    &=(l-\lambda)\cdot\Big[ y s(\lambda)+\lambda y m_1(\lambda)+o(1)\Big]~.
  \end{align*}
  Therefore, we have
  \begin{align*}
    \quad\frac {\lambda}{n} Z_1[A(l)-A(\lambda)]Z^*_1 \rightarrow (l-\lambda)\cdot( y \lambda s(\lambda)+\lambda^2 y m_1(\lambda))\cdot I_M~.
  \end{align*}

  \noindent {\bf Proof of \eqref{x3}:}\\
  First recall the fact that
  $
  \cov (X_1)=U\left(\begin{array}{c}
      a_1 \quad\quad\\
      \quad\ddots \quad\\
      \quad \quad a_k
    \end{array}\right)U^{*}~
  $
  and $X_1$ is independent of $B$. Using Lemma \ref{f1},we have
  \begin{align*}
    \frac 1TX_1(B(l)-B(\lambda))X^*_1\rightarrow  \frac 1T\e \tr(B(l)-B(\lambda))\cdot U\left(\begin{array}{c}
        a_1 \quad\quad\\
        \quad\ddots \quad\\
        \quad \quad a_k
      \end{array}\right)U^{*}~.
  \end{align*}
  The part
  \begin{align*}
    &\quad\frac 1T\e \tr(B(l)-B(\lambda))\\
    &=\frac 1T\e\tr\left\{X^*_2\bigg[(l I_p-S)^{-1}-(\lambda I_p-S)^{-1}\bigg]\Big(\frac 1n Z_2Z^*_2\Big)^{-1}\frac{1}{T}X_2\right\}\\
    &=(l-\lambda)\cdot\left[-\frac 1T \e \tr \left\{(\lambda I_p-S)^{-2}S\right\}+o(1)\right]\\
    &=(l-\lambda)\cdot\left[-c\int \frac{x}{(\lambda-x)^2}dF_{c,y}(x)+o(1)\right]\\
    &=(l-\lambda)\cdot(-cm_3(\lambda)+o(1))~.
  \end{align*}
  Therefore, we have
  \begin{align*}
    \frac 1TX_1(B(l)-B(\lambda))X^*_1\rightarrow -c(l-\lambda)m_3(\lambda)\cdot U \left(\begin{array}{c}
        a_1 \quad\quad\\
        \quad\ddots \quad\\
        \quad \quad a_k
      \end{array}\right)U^{*}~.
  \end{align*}

  \noindent {\bf Proof of \eqref{x4} and \eqref{x5}:}
  \eqref{x4} and \eqref{x5} are derived simply due to the fact that $\cov (X_1,Z_1)={\bf 0}_{M \times M}$~.

  \noindent The proof of Lemma \ref{asl} is complete.
\end{proof}

\begin{lemma}\label{lemm3}
  Define
  \begin{align*}
    \tilde{R}_n(\lambda_i):=\begin{pmatrix}
      Z_1 & S_1 \\
    \end{pmatrix}\left(\begin{array}{cc}
        \frac{\lambda_i \sqrt p A(\lambda_i)}{ n} & \frac{\lambda_i  \sqrt {a_ip} C(\lambda_i)}{ n}\\[2mm]
        \frac{\lambda_i \sqrt {a_ip} D(\lambda_i)}{ T} & \frac{-a_i \sqrt p B(\lambda_i)}{T}
      \end{array}\right) \begin{pmatrix}
      Z^*_1 \\[3mm]
      S^*_1 \\
    \end{pmatrix}-\E [\cdot]~.
  \end{align*}
  then
  $\tilde{R}_n(\lambda_i)$ weakly converges to a $M \times M$ symmetric random matrix $R(\lambda_i)=(R_{mn})$, which is made  with independent Gaussian entries of  mean zero and variance
  \begin{align*}
    \var (R_{mn})=\left\{\begin{array}{ll}
        2 \theta_i +(v_4-3)\omega_i~,& m=n~,\\
        \theta_i~, & m\neq n~,
      \end{array}\right.
  \end{align*}
  where
  \begin{align*}
    &\omega_i=\frac{a_i^2(a_i+c-1)^2(c+y)}{(a_i-1)^2}~,\\
    &\theta_i=\frac{a_i^2(a_i+c-1)^2(cy-c-y)}{-1+2a_i+c+a_i^2(y-1)}~.
  \end{align*}
\end{lemma}
\begin{proof}
  \noindent Since $Z_1$ and $S_1$ are independent, having the same first four moments, both are made with i.i.d. components, we can now view $\left(\begin{array}{cc}
      Z_{1}&S_{1}\end{array}\right)$ as a $M \times (n+T)$ table $\xi$, made with i.i.d elements of mean 0 and variance 1. Besides, we can rewrite the expression of $A(\lambda)$, $B(\lambda)$, $C(\lambda)$ and $D(\lambda)$ as follows:
  \begin{align*}
    A(\lambda)&=I_n-Z^*_2\Big(\lambda \cdot \frac 1n Z_2Z^*_2-\frac 1T X_2X^*_2\Big)^{-1}\frac{\lambda}{n}Z_2~,\nonumber\\
    B(\lambda)&=I_T+X^*_2\Big(\lambda \cdot \frac 1n Z_2Z^*_2-\frac 1T X_2X^*_2\Big)^{-1}\frac{1}{T}X_2~,\nonumber\\
    C(\lambda)&=Z^*_2\Big(\lambda \cdot \frac 1n Z_2Z^*_2-\frac 1T X_2X^*_2\Big)^{-1}\frac{1}{T}X_2~,\nonumber\\
    D(\lambda)&=X^*_2\Big(\lambda \cdot \frac 1n Z_2Z^*_2-\frac 1T X_2X^*_2\Big)^{-1}\frac{1}{n}Z_2~.
  \end{align*}
  It holds
  \begin{align*}
    A(\lambda)^*=A(\lambda)~, ~B(\lambda)^*=B(\lambda)~, ~T\cdot C(\lambda)^*=n\cdot D(\lambda)~,
  \end{align*}
  therefore, the matrix
  \begin{align*}
    \left(\begin{array}{cc}
        \frac{\lambda_i \sqrt p A(\lambda_i)}{ n} & \frac{\lambda_i  \sqrt {a_ip} C(\lambda_i)}{ n}\\[2mm]
        \frac{\lambda_i \sqrt {a_ip} D(\lambda_i)}{ T} & \frac{-a_i \sqrt p B(\lambda_i)}{T}
      \end{array}\right)
  \end{align*}
  is symmetric. Define
  \begin{align}\label{anl}
    A_n(\lambda_i)=\sqrt {n+T}\cdot\left(\begin{array}{cc}
        \frac{\lambda_i \sqrt p A(\lambda_i)}{ n} & \frac{\lambda_i  \sqrt {a_ip} C(\lambda_i)}{ n}\\[2mm]
        \frac{\lambda_i \sqrt {a_ip} D(\lambda_i)}{ T} & \frac{-a_i \sqrt p B(\lambda_i)}{T}
      \end{array}\right)~.
  \end{align}Now we can apply the results in \cite{BaiYao08} (Proposition 3.1 and Remark 1), which says that $\tilde{R}_n(\lambda_i)$ weakly converges to a $M \times M$ symmetric random matrix $R(\lambda_i)=(R_{mn})$, which is made  with i.i.d. Gaussian entries of  mean zero and variance
  \begin{align*}
    \var (R_{mn})=\left\{\begin{array}{ll}
        2 \theta_i +(v_4-3)\omega_i~,& m=n~,\\
        \theta_i~, & m\neq n~,
      \end{array}\right.
  \end{align*}
  The following is devoted to the calculation of the values of $\theta_i$ and $\omega_i$.

  \noindent {\bf Calculating of $\theta_i$:}
  From the definition of $\theta$ (see \cite{BaiYao08} for details), we have
  \begin{align}\label{aa1}
    \theta_i&=\lim \frac{1}{n+T} \tr A^2_n(\lambda_i)\nonumber\\
    &=\lim  \tr\left(\begin{array}{cc}
        \frac{\lambda_i \sqrt p A(\lambda_i)}{ n} & \frac{\lambda_i  \sqrt {a_ip} C(\lambda_i)}{ n}\\[2mm]
        \frac{\lambda_i \sqrt {a_ip} D(\lambda_i)}{ T} & \frac{-a_i \sqrt p B(\lambda_i)}{T}
      \end{array}\right)\left(\begin{array}{cc}
        \frac{\lambda_i \sqrt p A(\lambda_i)}{ n} & \frac{\lambda_i  \sqrt {a_ip} C(\lambda_i)}{ n}\\[2mm]
        \frac{\lambda_i \sqrt {a_ip} D(\lambda_i)}{ T} & \frac{-a_i \sqrt p B(\lambda_i)}{T}
      \end{array}\right)\nonumber\\
    &=\lim  \tr \left(\begin{array}{cc}
        \frac{p\lambda_i^2}{n^2}A^2(\lambda_i)+\frac{\lambda_i^2 a_ip}{nT}C(\lambda_i)D(\lambda_i) & \star\\[2mm]
        \star & \frac{\lambda_i^2 a_ip}{nT}D(\lambda_i)C(\lambda_i)+\frac{a^2_ip}{T^2} B^2(\lambda_i)
      \end{array}\right)\nonumber\\
    &=\lim \left[\frac{p\lambda_i^2}{n^2}\tr A^2(\lambda_i)+\frac{2\lambda_i^2 a_ip}{nT}\tr C(\lambda_i)D(\lambda_i)+\frac{a^2_ip}{T^2}\tr B^2(\lambda_i)\right]~.
  \end{align}
  \begin{align}\label{aa2}
    \tr A^2(\lambda_i)&=\tr \left[I_n+Z^*_2(\lambda_i I_p-S)^{-1}\left(\frac 1n Z_2Z^*_2\right)^{-1}\frac {\lambda_i}{n} Z_2Z^*_2(\lambda_i I_p-S)^{-1}\left(\frac 1n Z_2Z^*_2\right)^{-1}\frac {\lambda_i}{n} Z_2\right.\nonumber\\
    &\quad \quad \left.-2Z^*_2(\lambda_i I_p-S)^{-1}\left(\frac 1n Z_2Z^*_2\right)^{-1}\frac {\lambda_i}{n} Z_2\right]\nonumber\\
    &=n+\lambda_i^2\tr (\lambda_i I_p-S)^{-2}-2\lambda_i \tr (\lambda_i I_p-S)^{-1}\nonumber\\
    &=n+p\lambda_i^2m_1(\lambda_i)+2p\lambda_i s(\lambda_i)~,
  \end{align}

  \begin{align}\label{aa3}
    \tr C(\lambda_i)D(\lambda_i)&=\tr \left\{Z^*_2(\lambda_i I_p-S)^{-1}\left(\frac 1n Z_2Z^*_2\right)^{-1}\frac 1T X_2X^*_2(\lambda_i I_p-S)^{-1}\left(\frac 1n Z_2Z^*_2\right)^{-1}\frac 1n Z_2\right\}\nonumber\\
    &=\tr (\lambda_i I_p-S)^{-1}S(\lambda_i I_p-S)^{-1}=pm_3(\lambda_i)
  \end{align}

  \begin{align}\label{aa4}
    \tr B^2(\lambda_i)&=\tr \left[I_T+X^*_2(\lambda_i I_p-S)^{-1}\left(\frac 1n Z_2Z^*_2\right)^{-1}\frac {1}{T} X_2X^*_2(\lambda_i I_p-S)^{-1}\left(\frac 1n Z_2Z^*_2\right)^{-1}\frac {1}{T} X_2\right.\nonumber\\
    &\quad \quad \left.+2X^*_2(\lambda_i I_p-S)^{-1}\left(\frac 1n Z_2Z^*_2\right)^{-1}\frac {1}{T} X_2\right]\nonumber\\
    &=T+\tr (\lambda_i I_p-S)^{-1}F(\lambda_i I_p-S)^{-1}S+2\tr (\lambda_i I_p-S)^{-1}S\nonumber\\
    &=T+pm_4(\lambda_i)+2p m_2(\lambda_i)~,
  \end{align}

  \noindent Combining \eqref{aa1}, \eqref{aa2}, \eqref{aa3} and \eqref{aa4}, we have
  \begin{align*}
    \theta_i&=\lambda_i^2y(1+y\lambda_i^2m_1(\lambda_i)+2y\lambda_i s(\lambda_i))+2\lambda_i^2 a_icym_3(\lambda_i)+a^2_ic(1+cm_4(\lambda_i)+2cm_2(\lambda_i))\\
    &=\frac{a^2_i(a_i+c-1)^2(cy-c-y)}{-1+2a_i+c+a^2_i(y-1)}.
  \end{align*}

  \noindent {\bf Calculating of $\omega_i$:}
  \begin{align}\label{w}
    \omega_i=\lim \frac{1}{n+T}\sum_{i=1}^{n+T}(A_n(\lambda_i)(i,i))^2=\lim \left[\sum_{i=1}^{n}\frac{\lambda_i^2p}{n^2}A^2(i,i)+\sum_{i=1}^{T}\frac{a^2_ip}{T^2}B^2(i,i)\right]~.
  \end{align}

  \noindent In the following, we will show that $A(i,i)$ and $B(i,i)$ both tend to some limits that is independent of $i$.
  \begin{align}\label{aaaa2}
    A(i,i)&=1-\bigg[Z^*_2\Big[\lambda_i I_p-\big(\frac 1n Z_2Z^*_2\big)^{-1}\frac 1T X_2X^*_2\Big]^{-1}\big(\frac 1n Z_2Z^*_2\big)^{-1}\frac{\lambda_i}{n}Z_2\bigg](i,i)\nonumber\\
    &=1-\frac{\lambda_i}{n}\bigg[Z^*_2\Big[\lambda_i\cdot \frac 1n Z_2Z^*_2-\frac 1T X_2X^*_2\Big]^{-1}Z_2\bigg](i,i)
  \end{align}
  If we denote $\eta_i$ as the $i$-th column of $Z_2$,  we have
  \begin{align*}
    \frac 1n Z_2Z^*_2=\frac 1n\begin{pmatrix}
      \eta_1 & \cdots & \eta_n \\
    \end{pmatrix}
    \cdot \begin{pmatrix}
      \eta^*_1 \\
      \vdots \\
      \eta^*_n \\
    \end{pmatrix}=\frac 1n \eta_i \eta^*_i+\frac 1n Z_{2i}Z^*_{2i}~,
  \end{align*}
  where $Z_{2i}$ is independent of $\eta_i$.
  Since
  \begin{align*}
    &\quad \Big(\lambda_i\cdot \frac 1n Z_2Z^*_2-\frac 1T X_2X^*_2\Big)^{-1}-\Big(\lambda_i\cdot \frac 1n Z_{2i}Z^*_{2i}-\frac 1T X_2X^*_2\Big)^{-1}\\
    &=
    -\Big(\lambda_i\cdot \frac 1n Z_2Z^*_2-\frac 1T X_2X^*_2\Big)^{-1}\frac {\lambda_i}{n} \eta_i \eta^*_i \Big(\lambda_i\cdot \frac 1n Z_{2i}Z^*_{2i}-\frac 1T X_2X^*_2\Big)^{-1}~,
  \end{align*}
  we have
  \begin{align}\label{aaaa1}
    \Big(\lambda_i\cdot \frac 1n Z_2Z^*_2-\frac 1T X_2X^*_2\Big)^{-1}=\frac{\Big(\lambda_i\cdot \frac 1n Z_{2i}Z^*_{2i}-\frac 1T X_2X^*_2\Big)^{-1}}{1+\frac {\lambda_i}{n} \eta^*_i\Big(\lambda_i\cdot \frac 1n Z_{2i}Z^*_{2i}-\frac 1T X_2X^*_2\Big)^{-1}\eta_i}~.
  \end{align}
  Bringing \eqref{aaaa1} into \eqref{aaaa2},
  \begin{align*}
    A(i,i)&=1-\frac{\lambda_i}{n}\eta^*_i\Big[\lambda_i\cdot \frac 1n Z_2Z^*_2-\frac 1T X_2X^*_2\Big]^{-1}\eta_i\nonumber\\
    &=1-\frac{\frac{\lambda_i}{n}\eta^*_i\Big(\lambda_i\cdot \frac 1n Z_{2i}Z^*_{2i}-\frac 1T X_2X^*_2\Big)^{-1}\eta_i}{1+\frac {\lambda_i}{n} \eta^*_i\Big(\lambda_i\cdot \frac 1n Z_{2i}Z^*_{2i}-\frac 1T X_2X^*_2\Big)^{-1}\eta_i}\nonumber\\
    &=\frac{1}{1+\frac {\lambda_i}{n} \eta^*_i\Big(\lambda_i\cdot \frac 1n Z_{2i}Z^*_{2i}-\frac 1T X_2X^*_2\Big)^{-1}\eta_i}~,
  \end{align*}
  whose denominator of \eqref{aaaa3} equals
  \begin{align}\label{aaaa3}
    1+\frac {\lambda_i}{n} \tr \Big(\lambda_i\cdot \frac 1n Z_{2i}Z^*_{2i}-\frac 1T X_2X^*_2\Big)^{-1}\eta_i\eta^*_i~.
  \end{align}
  Since $\eta_i$ is independent of $\Big(\lambda_i\cdot \frac 1n Z_{2i}Z^*_{2i}-\frac 1T X_2X^*_2\Big)^{-1}$,  \eqref{aaaa3}
  converges to the value $1+\lambda_i y \cdot \frac{1}{a_i+c-1}$ according to Lemma \ref{l1}.
  Therefore, we have
  \begin{align}
    A(i,i)\rightarrow \frac{1}{1+\lambda_i y\cdot \frac{1}{a_i+c-1}}~,
  \end{align}
  which is independent of the choice of $i$.

  \noindent For the same reason, we have
  \begin{align}\label{b1}
    B(i,i)&=1+\bigg[X^*_2\Big[\lambda_i I_p-\big(\frac 1n Z_2Z^*_2\big)^{-1}\frac 1T X_2X^*_2\Big]^{-1}\big(\frac 1n Z_2Z^*_2\big)^{-1}\frac{1}{T}X_2\bigg](i,i)\nonumber\\
    &=1+\bigg[X^*_2\Big[\lambda_i \cdot \frac 1n Z_2Z^*_2-\frac 1T X_2X^*_2\Big]^{-1}\frac{1}{T}X_2\bigg](i,i)~.
  \end{align}
  If we denote $\delta_i$ as the $i$-th column of $X_2$, then we have
  \begin{align*}
    \frac 1T X_2X^*_2=\frac 1T\begin{pmatrix}
      \delta_1 & \cdots & \delta_T \\
    \end{pmatrix}
    \cdot \begin{pmatrix}
      \delta^*_1 \\
      \vdots \\
      \delta^*_T \\
    \end{pmatrix}=\frac 1T \delta_i \delta^*_i+\frac 1T X_{2i}X^*_{2i}~,
  \end{align*}
  and
  \begin{align*}
    &\quad \Big(\lambda_i\cdot \frac 1n Z_2Z^*_2-\frac 1T X_2X^*_2\Big)^{-1}-\Big(\lambda_i\cdot \frac 1n Z_{2}Z^*_{2}-\frac 1T X_{2i}X^*_{2i}\Big)^{-1}\\
    &=
    \Big(\lambda_i\cdot \frac 1n Z_2Z^*_2-\frac 1T X_2X^*_2\Big)^{-1}\frac {1}{T} \delta_i \delta^*_i \Big(\lambda_i\cdot \frac 1n Z_{2}Z^*_{2}-\frac 1T X_{2i}X^*_{2i}\Big)^{-1}~.
  \end{align*}
  So we have
  \begin{align}\label{b2}
    \Big(\lambda_i\cdot \frac 1n Z_2Z^*_2-\frac 1T X_2X^*_2\Big)^{-1}=\frac{\Big(\lambda_i\cdot \frac 1n Z_{2}Z^*_{2}-\frac 1T X_{2i}X^*_{2i}\Big)^{-1}}{1-\frac 1 T\delta^*_i\Big(\lambda_i\cdot \frac 1n Z_{2}Z^*_{2}-\frac 1T X_{2i}X^*_{2i}\Big)^{-1}\delta_i}~.
  \end{align}
  Combine \eqref{b1} and \eqref{b2}, we have
  \begin{align}
    B(i,i)&=1+\delta^*_i\Big[\lambda_i \cdot \frac 1n Z_2Z^*_2-\frac 1T X_2X^*_2\Big]^{-1}\frac{1}{T}\delta_i\nonumber\\
    &=1+\frac{\frac 1 T\delta^*_i\Big(\lambda_i\cdot \frac 1n Z_{2}Z^*_{2}-\frac 1T X_{2i}X^*_{2i}\Big)^{-1}\delta_i}{1-\frac 1 T\delta^*_i\Big(\lambda_i\cdot \frac 1n Z_{2}Z^*_{2}-\frac 1T X_{2i}X^*_{2i}\Big)^{-1}\delta_i}\nonumber\\
    &=\frac{1}{1-\frac 1 T\delta^*_i\Big(\lambda_i\cdot \frac 1n Z_{2}Z^*_{2}-\frac 1T X_{2i}X^*_{2i}\Big)^{-1}\delta_i}~.
  \end{align}
  Using the independence between $\delta_i$ and $\Big(\lambda_i\cdot \frac 1n Z_{2}Z^*_{2}-\frac 1T X_{2i}X^*_{2i}\Big)^{-1}$ and Lemma \ref{l1} again, we have
  \[
  \frac 1 T\delta^*_i\Big(\lambda_i\cdot \frac 1n Z_{2}Z^*_{2}-\frac 1T X_{2i}X^*_{2i}\Big)^{-1}\delta_i\rightarrow c\cdot \frac{1}{a_i+c-1}~.
  \]

  \noindent Therefore, we have
  \[
  B(i,i)\rightarrow \frac{1}{1-\frac{c}{a_i+c-1}}~,
  \]
  which is also independent of the choice of $i$.

  \noindent Finally, taking the definition of $\omega_i$ in \eqref{w} into consideration, we have
  \begin{align}
    \omega_i=\frac{\lambda_i^2y}{\left(1+y\lambda_i \cdot \frac{1}{a_i+c-1}\right)^2}+\frac{a^2_ic}{\left(1-\frac{c}{a_i+c-1}\right)^2}=\frac{a^2_i(a_i+c-1)^2(c+y)}{(a_i-1)^2}~.
  \end{align}


  \noindent The proof of Lemma \ref{lemm3} is complete.

\end{proof}

\end{document}